\documentclass[11pt]{amsart}
\usepackage{a4wide,amsfonts,amsmath,amssymb,amsthm,epsfig,cite,graphicx,hyperref,
color, esint,fancyhdr, enumerate, latexsym,amsrefs}

\newtheorem{thm}{Theorem}[section]
\newtheorem{prop}[thm]{Proposition}
\newtheorem{lem}[thm]{Lemma}
\newtheorem{cor}[thm]{Corollary}
\theoremstyle{definition}

\newcommand{\ep}{\epsilon}
\newcommand{\mbb}{\mathbb}

\newcommand{\ov}{\overline}

\newcommand{\pa}{\partial}
\newcommand{\mf}{\mathbb}
\newcommand{\om}{\omega}
\newcommand{\Om}{\Omega}
\newcommand{\al}{\alpha}

\newcommand{\z}{\zeta}

\newcommand{\ti}{\tilde}
\newcommand{\De}{\Delta}

\renewcommand{\Re}{\operatorname{Re}}

\newcommand{\Aut}{\operatorname{Aut}}
\numberwithin{equation}{section}

\hypersetup{
    colorlinks,
    citecolor=blue,
    filecolor=black,
    linkcolor=red,
    urlcolor=black
}

\title{Limits of an increasing sequence of complex manifolds}
\keywords{Union problem, Kobayashi hyperbolic, Kobayashi corank one, Levi corank one domains}
\subjclass{Primary: 32F45; Secondary: 32H15, 32H20}
\author{G. P. Balakumar, Diganta Borah, Prachi Mahajan, and Kaushal Verma}

\address{G. P. Balakumar: Department of Mathematics, Indian Institute of Technology Palakkad, 678557, India}
\email{gpbalakumar@gmail.com}

\address{Diganta Borah: Indian Institute of Science Education and Research, Pune  411008, India}
\email{dborah@iiserpune.ac.in}

\address{Prachi Mahajan: Department of Mathematics, Indian Institute of Technology Bombay, Powai, Mumbai 400076, India}
\email{prachi.mjn@iitb.ac.in}

\address{Kaushal Verma: Department of Mathematics, Indian Institute of Science, Bangalore 560 012, India}
\email{kverma@iisc.ac.in}

\begin{document}
\maketitle
\begin{abstract}
Let $M$ be a complex manifold which admits an exhaustion by open subsets $M_j$ each of which is biholomorphic to a fixed domain $\Omega \subset \mbb C^n$. The main question addressed here is to describe $M$ in terms of $\Omega$. Building on work of Fornaess--Sibony, we study two cases namely, $M$ is Kobayashi hyperbolic and the other being the corank one case in which the Kobayashi metric degenerates along one direction. When $M$ is Kobayashi hyperbolic, its complete description is obtained when $\Omega$ is one of the following domains -- (i) a smoothly bounded Levi corank one domain, (ii) a smoothly bounded  convex domain, (iii) a strongly pseudoconvex polyhedral domain in $\mbb C^2$, or (iv) a simply connected domain in $\mbb C^2$ with generic piecewise smooth Levi-flat boundary. With additional hypotheses, the case when $\Omega$ is the minimal ball or the symmetrized polydisc in $\mbb C^n$ can also be handled. When the Kobayashi metric on $M$ has corank one and $\Omega$ is either of (i), (ii) or (iii) listed above, it is shown that $M$ is biholomorphic to a locally trivial fibre bundle with fibre $\mbb C$ over a holomorphic retract of $\Omega$ or that of a limiting domain associated with it. Finally, when $\Omega = \Delta \times \mbb B^{n-1}$, the product of the unit disc $\Delta \subset \mbb C$ and the unit ball $\mbb B^{n-1} \subset \mbb C^{n-1}$, a complete description of holomorphic retracts is obtained. As a consequence, if $M$ is Kobayashi hyperbolic and $\Omega = \Delta \times \mbb B^{n-1}$, it is shown that $M$ is biholomorphic to $\Omega$. Further, if the Kobayashi metric on $M$ has corank one, then $M$ is globally a product; in fact, it is biholomorphic to $Z \times \mbb C$, where $Z \subset \Omega = \Delta \times \mbb B^{n-1}$ is a holomorphic retract.
\end{abstract}

\maketitle


\section{Introduction}

Let $M$ be a complex manifold which is the union of an increasing sequence of open subsets $M_j$ each of which is biholomorphic to a fixed domain $\Om \subset \mbb{C}^n$. Then it is of interest to describe $ M $ in terms of $\Om$. In what follows, this problem will be referred to as the \textit{union problem}.

\medskip

Among the simplest cases, $ \Omega = \mathbb{B}^n $ was considered by Forn\ae ss and Stout in \cite{For-Sto} who showed that $ M $ is biholomorphic to $ \mathbb{B}^n $ if $ M $ is an open subset of a taut complex manifold. Subsequently in \cite{FS1981}, Forn\ae ss and Sibony investigated the union problem using the infinitesimal Kobayashi metric $F_{M}(p,v)$ at $ p $ of $ M $ in the tangent direction $ v \in T_pM $ (here and henceforth by $T_pM$ we mean the holomorphic tangent space $T_p^{1,0}M$). Assume that $ \Omega $ is hyperbolic and $ \Omega/ \Aut(\Omega) $ is compact. Here, $ \Aut(\Omega) $ denotes the holomorphic automorphism group of $ \Omega $ with the standard compact open topology. It was shown that if there is a $ p \in M $ such that $ F_M(p, v) \neq 0 $ whenever $ v \neq 0 $, then $ M $ is biholomorphic to $ \Omega $. Furthermore, under the same assumptions on $ \Omega $, it was also proved that the zero set of $F_M(p, \cdot): T_pM \to \mf{R}$, which is a complex vector subspace of $T_pM$, has dimension independent of $p$; and if this constant dimension, called the \textit{corank} of $F_M$, is one, then there exists a closed complex submanifold $ Z $  of $ \Omega $ such that $ M $ is biholomorphic to a locally trivial holomorphic fibre bundle over $ Z $ with fibre $ \mathbb{C} $. Behrens in \cite{Behrens} further extended the results of Forn\ae ss and Sibony: If $ \Omega \subset \mathbb{C}^n $ is $ C^{2}$-smooth strongly pseudoconvex and $ M $ is hyperbolic, then $ M $ is biholomorphic either to $ \Omega $ or to the unit ball $ \mathbb{B}^n \subset \mathbb{C}^n $. Moreover, analogous results as in  \cite{FS1981} were obtained for strongly pseudoconvex domains $ \Omega $ when $ M $ is non-hyperbolic. Behrens' proof in the case when $ M $ is hyperbolic crucially used Pinchuk's scaling techniques and the fact that the model domain at a strongly pseudoconvex boundary point is the ball. Related work on the union problem can be found in \cite{DS} and \cite{Liu}.

\medskip

The first objective of this work is to study the union problem for a broader class of domains $ \Omega $, more specifically the Levi corank one domains in $\mf{C}^n$, smoothly bounded convex domains in $\mf{C}^n$, the strongly pseudoconvex polyhedral domains in $\mf{C}^2$, the minimal ball, simply connected domains in $\mf{C}^2$ with generic piecewise smooth Levi-flat boundaries, and the symmetrized polydisc. The definitions of all of them will be given later in Section~3. Note that no assumption is made about the quotient $ \Omega/ \Aut(\Omega) $ and this is indeed non-compact also for the minimal ball, whose boundary is not smooth but is devoid of nontrivial analytic varieties; such non-smooth convex domains are dealt with here as well. When $\Om$ is one of the first four, then (it is well known and will be explained later that) $\Om$ is amenable to scaling, which when applied to a sequence $\{p^j\}$ in $\Om$ converging to a boundary point $p^0$, yields a limit domain $\Om_{\infty}$. When $\Om$ is a Levi corank one domain or a convex finite type domain, then $\Om_{\infty}$ turns out to be a polynomial domain. Before moving on, a word about notation -- the use of the same symbol $\Omega_{\infty}$ to denote the limit domains in all these cases will not lead to any confusion since each of the various classes of domains listed above will be handled separately later.

\begin{thm} \label{T1}
Assume that in the union problem, $ M$ is a hyperbolic manifold. 
\begin{enumerate}
\item[(i)] If $ \Omega \subset \mf{C}^n $ is a bounded Levi corank one domain, then $ M $ is biholomorphic either to $ \Omega $ or to a limiting domain of the form
\begin{equation*}
 \Omega_{\infty}=\Big\{ z \in \mathbb{C}^n : 2 \Re z_n + P_{2m} \left(z_1, \overline{z}_1 \right) + \sum_{j=2}^{n-1} \vert z_j \vert^2 < 0 \Big\},
\end{equation*}
where $m\geq 1$ is a positive integer and $ P_{2m} \left(z_1, \overline{z}_1 \right) $ is a subharmonic polynomial of degree at most $ 2m $ without any harmonic terms.
 
\item[(ii)] If $ \Omega \subset \mf{C}^n $ is a smoothly bounded convex domain
then $ M $ is biholomorphic either to $ \Omega $ or to a limiting domain $\Om_{\infty}$ associated to $\Omega$. 

\item[(iii)] If $n=2$ and $ \Omega \subset \mf{C}^2$ is a strongly pseudoconvex polyhedral domain, then $ M $ is biholomorphic either to $ \Omega $ or to $ \Omega_{\infty}$, where $ \Omega_{\infty} $ is a limiting domain associated to $ \Omega $.

\item[(iv)] Let $ \Omega =\left\{z \in \mf{C}^n: \frac{1}{2}\left( \sum_{j=1}^{n} \vert z_j \vert^2 + \big\vert \sum_{j=1}^{n} z_j^2 \big\vert \right)<1\right\}$ be the minimal ball {or more generally any bounded convex domain whose boundary (is not necessarily smooth but) does not contain nontrivial complex analytic varieties. If $M$ is a priori known to be complete hyperbolic, then $M \simeq \Omega$}.

\item[(v)] If $n=2$ and $ \Omega \subset \mf{C}^2 $ is a simply connected domain with generic piecewise $ C^{\infty}$-smooth Levi-flat boundary, then $ M$ is biholomorphic either to $ \Omega $ or to the unit bidisc $ \Delta^2 \subset \mf{C}^2$.

\item[(vi)] If $M$ is a bounded domain in $\mf{C}^n$ and $ \Omega$ is the symmetrized polydisc, then either $ M $ is biholomorphic to $ \Omega $ or $ M $ admits a proper holomorphic correspondence to $ \Delta^n $, the unit polydisc in $ \mathbb{C}^n $, with each fibre having cardinality at most $ n! $.
 
\end{enumerate}
\end{thm}

\begin{thm} \label{T2}
Assume that in the union problem, $ M $ is a non-hyperbolic manifold. Then under any of the hypothesis as in Theorem~\ref{T1} (i)-(iii),
\begin{enumerate}

\item[(i)] The dimension of $\{v\in T_pM: F_M(p,v)=0\}$ is independent of $p$.

\item [(ii)] The zero set of $F_M$ is a vector bundle over $ M $. 

\item [(iii)] If the corank of $F_M$ is one, then $ M $ is biholomorphic to a locally trivial holomorphic fibre
bundle with fibre $ \mathbb{C} $ over a retract of $ \Omega $ or that of a limiting domain $ \Omega_{\infty} $ associated to $\Om$.
\end{enumerate}
\end{thm}
Several remarks are in order. The dichotomy in Theorem~\ref{T1} that $ M $ is biholomorphic either to $ \Omega $ or to a domain $ \Omega_{\infty} $ arises from the two cases that need to be considered -- if $ \psi^j : M_j \rightarrow \Omega $ are given biholomorphisms and $ z^0 \in M $ is a given fixed point then -- first, when the orbit $ \{ \psi^j (z^0) \} $ is compactly contained in $ \Omega $ and second, when some subsequence of $ \{ \psi^j (z^0) \} $ (which we continue to denote by the same symbols) accumulates at a boundary point $ p^0 \in \partial \Omega $. In the latter case, scaling the domain $ \Omega $ along $ \{ \psi^j (z^0) \} $ yields a sequence of domains $ \Omega_j $ that converge to $ \Omega_{\infty} $ in the local Hausdorff sense. Observe that the scaled domains $ \Omega_j $, thereby their limit $\Omega_\infty$ in particular, depends only on the geometry of $\partial \Omega$ near $p^0$. This allows us to extend Theorem~\ref{T1} with apparently far less assumptions as in Theorem~\ref{T3} mentioned in the last section.

\medskip

Next, notice that $ \partial \Omega $ can be \textit{a priori} be of infinite type in Theorem~\ref{T1} (ii). Therefore, for smoothly bounded convex domains $ \Omega $, in the case $ \{ \psi^j (z^0) \} $ converges to $ p^0 \in \partial \Omega $, there are further two possibilities 
to be analysed---(I) $ \partial \Omega $ is of finite type near $ p^0 $, and (II) $ \partial \Omega $ is of infinite type near $ p^0 $. As before, the general strategy is to scale the domain $ \Omega $ with respect to the sequence $ \{ \psi^j (z^0) \} $ to get rescaled domains 
$ \Omega_j $. The associated limiting domain $ \Omega_{\infty} $ is hyperbolic convex in both the cases (1) and (2).

\medskip

Next, it follows from Theorem \ref{T1} that the symmetrized polydisc $ \mathbb{G}_n $ cannot exhaust a strongly pseudoconvex domain $D$. If it did, then either $D$ would be biholomorphic to $ \mathbb{G}_n $ or there would be a proper correspondence between $D$ and $ \Delta^n $---both these scenarios  would lead to a contradiction (see \cite{Huc-Orm78}, \cite{Ris64}).

\begin{cor}
A $ C^2 $-smooth strongly pseudoconvex domain $ D \subset \mathbb{C}^n $ cannot be the union of an increasing sequence of open subsets $ D_j $, each of which is biholomorphic to the symmetrized polydisc $ \mathbb{G}_n $.
\end{cor}

Let us now take $\Om \subset \mf{C}^2$ to be a $C^{\infty}$-smoothly bounded strongly convex domain. Theorem~\ref{T1} then shows that $M $ is biholomorphic to either $ \Om$ or $\mf{B}^2$ in case $M$ is assumed to be hyperbolic. On the other hand, if $M$ is non-hyperbolic and the corank of $F_M$ is one, then by Theorem~\ref{T2}, $ M $ is biholomorphic to a locally trivial holomorphic fibre bundle with fibre $ \mathbb{C} $ over a retract $ Z $ of $ \Omega $ or $\mf{B}^2$ of dimension one. By Lempert's work \cite{Lem82}, one-dimensional holomorphic retracts in strongly convex domains are exactly the extremal maps for the Kobayashi metric. Hence, there is a holomorphic map $f : \De \to \Om$ or $f : \De \to \mf{B}^2$ such that $f(\De)=Z$ and $f$ is an extremal map for the Kobayashi metric. By Lempert's work again \cite{Lem81}, $f$ is a complex geodesic and hence is an embedding, and thus $Z $ is biholomorphic to $ \De$. In particular, $Z$ admits solutions to both the additive and multiplicative Cousin problems, and thus by the arguments used in the proving Corollary 4.8 of \cite{FS1981}, it follows that $M $ is biholomorphic to $ \De \times \mf{C}$. This replaces the computation in Lemma~4.7 of \cite{FS1981} wherein $Z$ is identified in case $\Om $ is $\mf{B}^n$ or $\De^n$. It is however rare for a domain to admit either holomorphic retracts or Kobayashi extremals. And in fact, it is quite surprising that $\mf{B}^n$ admits holomorphic retracts of all possible codimensions---see \cite{Suf74}.
\begin{cor}\label{sc-dim2}
Assume that in the union problem, $n=2$ and $\Om \subset \mf{C}^2$ is a $C^\infty$-smoothly bounded strongly convex domain. If $M$ is hyperbolic, then $M $ is biholomorphic to either $ \Om$ or $\mf{B}^2$. If $M$ is non-hyperbolic and the corank of $F_M$ is one, then $M$ is biholomorphic to $ \De \times \mf{C}$.
\end{cor}

Observe that if $\Om \subset \mf{C}^2$ is a $C^{\infty}$-smooth small perturbation of $\mf{B}^2$, then $\Om$ is a $C^{\infty}$-smoothly bounded strongly convex domain that is not biholomorphic to $ \mf{B}^2$. Hence, Corollary \ref{sc-dim2} generalises the above mentioned result of \cite{FS1981} in dimension two. 

\medskip

As to what happens in case $\Omega$ is not strongly convex, particularly in the case when $M$ is not hyperbolic seems rather challenging in 
general. To study a simple case not covered by the convex domains already dealt with by the above corollary or Theorem~\ref{T2}, we take $\Omega = \Delta \times \mathbb{B}^{n-1}$. By the main theorem of \cite{FS1981}, describing $M$ leads to determining the retracts of $\Delta \times \mathbb{B}^{n-1}$. The following result gives a description of all possible holomorphic retracts of $ \Delta \times \mathbb{B}^{n-1} $. Note that it suffices to determine the retracts of $ \Delta \times \mathbb{B}^{n-1}$ that contain the origin, from which the general case follows using the automorphisms of $ \Delta \times \mathbb{B}^{n-1} $. Recall that ${\rm Aut}(\Delta \times \mathbb{B}^{n-1}) $ is isomorphic to $ {\rm Aut}(\Delta) \times {\rm Aut}(\mathbb{B}^{n-1})$ and hence every automorphism of $ \Delta \times \mathbb{B}^{n-1}$ is a linear fractional transformation. 

\begin{thm}\label{retracts}
Let $ Z $ be a holomorphic retract of $ \Delta \times \mathbb{B}^{n-1}$. Assume that $ Z $ contains the origin, 
$ Z \neq \{0\} $ and $ Z \neq \Delta \times \mathbb{B}^{n-1}$. Then $ Z $ is given as one of the following:

\begin{itemize}
 \item [(i)] $ Z=\left \{ (w,c_1w,\ldots, c_{n-1} w) : w \in \Delta \right\} $ where  $(c_1,c_2, \ldots,c_{n-1}) \in \partial \mathbb{B}^{n-1}$,

\item [(ii)] $Z$ is the graph of a $\mathbb{B}^{n-1}$-valued holomorphic mapping of $\Delta$,

\item [(iii)] $Z$ is the graph of a $\Delta$-valued holomorphic function over a complex linear subspace of $ \mathbb{B}^{n-1}$,

\item [(iv)] $Z$ is the intersection of a linear subspace with $ \Delta \times \mathbb{B}^{n-1}$ of complex dimension at least two.
\end{itemize} 
\end{thm}

Let us continue with the observation that identifying retracts is central to the union problem.
 Recall the following facts about retracts: Let $S \subset \Om$ be a topological space and $S$ a retract of $\Om$. By writing $i, r$ for the inclusion and retraction respectively, we get
\[
S \xrightarrow{i} \Om \xrightarrow{r} S
\]
from which $r \circ i=1_S$, and hence $r_{*} i_{*}=1$ at the level of fundamental groups. This shows that
\[
i_{*}: \pi_1(S) \to \pi_1(\Om)
\]
is injective (and $r_*$ is surjective--but we will not need this). Therefore, if $\Om$ is simply connected, then $S$ is also simply connected.

\begin{cor}
For $n=2$, under the hypothesis of Theorem~\ref{T2} and with the additional property that $\pi_1(\Om)=0$, the holomorphic retract $Z \subset \Om$ given by this theorem must be simply connected and hence $Z $ is biholomorphic to $ \De$ (since $\Om$ is bounded).
\end{cor}

This approach circumvents the need for complex geodesics and shows that retracts are always simply connected if $\Om$ is so. In particular, when $\Om \subset \mf{C}^2$ is simply connected, every retract is equivalent to the disc $\De$. This argument can be applied in several cases: egg domains in $\mf{C}^2$ and the symmetrized bidisc (which is known to be contractible) included. Thus, if $\Om \subset \mathbb{C}^2$ (in the union problem) is a simply connected domain and amenable to scaling, then $
M$ is biholomorphic to $\De \times \mf{C}$ if $F_M$ has corank one.

\medskip

Let us now turn to the question: What happens if $\Om$ has non-trivial topology? As an example, let
\begin{equation}\label{ex-1}
\Om=\left\{(z,w) \in \mf{C}^2 : \vert z \vert^2 + \vert w \vert^2 + \frac{1}{\vert w\vert^2}<R \right\}
\end{equation}
for $R>2$. Then $\Om$ is a non-contractible, smooth strongly pseudoconvex domain. Also, $\Om$ deformation retracts to $\Om \cap (\{0\} \times \mf{C})$ which is a planar domain with fundamental group $\mf{Z}$. Hence $\pi_1(\Om)=\mf{Z}$. Theorem~\ref{T2} shows that if this $\Om$ is taken to be the model domain, then in case $F_M$ has corank $1$, $M$ is biholomorphic to a fibre bundle over a retract $Z \subset \Om$ with fibre $\mf{C}$. Now,
\[
i_{*} : \pi_1(Z) \to \pi_1(\Om)=\mf{Z}
\]
is injective and hence $\pi_1(Z)=0$ or $\mf{Z}$. In the former case $Z$ is biholomorphic to $\De$ (in which case we are back to the old situation). When $\pi_1(Z)=\mf{Z}$, then $Z$ is biholomorphic to either the punctured plane $\mf{C}^{*}$, or the punctured disc $\De^{*}$, or an annulus. As $Z$ is bounded, $\mf{C}^{*}$ is not possible. So $Z$ is biholomorphic to either $\De^{*}$ or an annulus. Both these domains admit a solution to the additive and multiplicative cousin problems and so it follows that $M$ is biholomorphic to $Z \times \mf{C}$.

\medskip

When the fundamental group of $\Om \subset \mf{C}^2$ increases in complexity, it is difficult in general to identify $\pi_1(Z)$ and argue as before. But the redeeming feature is that $Z$ is a non-compact Riemann surface and hence admits solutions to both the Cousin problems. So when $\Om \subset \mf{C}^2$ is amenable to scaling, for example, as in Theorem~\ref{T2}, with non-trivial topology, then $M$ is biholomorphic to $Z \times \mf{C}$ if $F_M$ has corank one.

\medskip

We conclude this article with the question if the limit $M$ in the union problem can be some special type of manifolds. First, let $\mf{P}^n$ be the standard $n$-dimensional complex projective space and $M_j \subset \mf{P}^n$ be an increasing sequence of open subsets each of which is biholomorphic to $\mathbb{B}^n$. The question that we would like to answer here is whether it is possible that $M$ be a quasiprojective variety in $\mathbb{P}^n$?

\begin{thm}\label{qp}
Let $M \subset \mf{P}^n$ be a quasiprojective variety that is the union of an increasing sequence of open subsets $M_j \subset \mf{P}^n$, each of which is biholomorphic to a domain $\Om \subset \mf{C}^n$. If  $\Om$ is bounded and $\Om/\Aut(\Om)$ is compact, then $M$ cannot be hyperbolic. Further, if $\Om=\mf{B}^n$ or $\De^n$, then the corank of $F_M$ is at least $2$.
\end{thm}
In particular, if $n=2$ and $\Om=\mf{B}^2$ or $\De^2$, then $M$ in the union problem cannot be quasiprojective. We do not know what happens in higher dimensions nor when the corank increases. 

\medskip

A similar question is if $M$ in the union problem can be the complement of a closed complete pluripolar set in a Stein manifold $X$.

\begin{thm}\label{cp}
Let $M$ be the union of an increasing sequence of open subsets $M_j$ in a Stein manifold $X$ of dimension $n$, each of which is biholomorphic ot $\mf{B}^n$. If $M=X\setminus P$ for some closed complete pluripolar set $P$ in $X$, then the corank of $F_M$ is at least one. In particular, $ M $ cannot be hyperbolic.
\end{thm}

\section{When $ M $ is hyperbolic}
In this section, we present the proof of Theorem \ref{T1}. Recall that $ M = \cup M_j $ is a hyperbolic manifold where $ M_j \subset M_{j+1} $ and $M_j$ is biholomorphic to $\Om$ for each $ j $. Let $ \psi^j : M_j \rightarrow \Omega $ be a biholomorphic mapping. Fixing $ z^0 \in M $, we may assume that $ z^0 \in M_j $ for all $ j$, and let $p^j:=\psi^j(z^0)$. There are two cases to be examined:
\begin{enumerate}
 \item[(a)] $ \{p^j \} $ is a relatively compact subset of $ \Omega $, and
 \item[(b)] $ \{ p^j \} $ has at least one limit point $ p^0 \in \partial \Omega $.
\end{enumerate}
In case (a), as the domain $\Om$ is taut, after passing to a subsequence, $\{\psi^j\}$ converges uniformly on compact subsets of $M$ to a holomorphic mapping $ \psi: M \rightarrow \Omega $ and it follows that $ M $ is biholomorphic to $ \Omega $ (see, for instance, Lemma 3.1 of \cite{MV2012}). Thus we are left with case (b).

\section{Proof of Theorem~\ref{T1} (i)-(iv)}
\subsection{Step I: Scaling method and the stability of the Kobayashi metric}
We briefly describe the scaling method for $\Om$ when it satisfies any of the hypotheses (i)-(iv) of Theorem~\ref{T1} and establish that in each of these cases we have

\begin{prop}\label{k-stability}
There exists a sequence of biholomorphic maps $A^j: \Om \to \Om_j$ where $\Om_j \subset \mf{C}^n$ are domains that converge in the local Hausdorff sense to a taut domain $\Om_{\infty} \subset \mf{C}^n$ and $q^j:=A^j(p^j) \to q^0$ for some $q^0 \in \Om_{\infty}$ which we will refer to as the \textit{base point}. Further, the family of maps $\ti \psi^j:=A^j \circ \psi^j: D_j \to \Om_j$ has a limit $\ti \psi:M \to \Om_{\infty}$ and the Kobayashi distances $d_{\Om_j}$ satisfies the following stability property on $\ti \psi(M)$:
\begin{equation}\label{E4}
\limsup_{j \rightarrow \infty} d_{\Omega_j} \big( \tilde{\psi} (z^1), \tilde{\psi} (z^2) \big) \leq d_{\Omega_{\infty}} \big( \tilde{\psi} (z^1), \tilde{\psi} (z^2) \big),
\end{equation}
for all $z^1,z^2 \in M$.
\end{prop}

\subsection*{Scaling a Levi corank one domain $ \Omega $:}
Let $\Omega$ be a domain in $\mathbb{C}^n$ such that a smooth piece 
$\Gamma \subset \partial \Omega$ forms a pseudoconvex (real) hypersurface in $\mathbb{C}^n$ of finite $ 1$-type (in the sense of D' Angelo, as always in this article). Then $p \in \Gamma $ is said to be a Levi corank one boundary point of $\Omega$ if the Levi form of $\Gamma$ has at least $ (n-2)$ positive eigenvalues at $p$. Also, every boundary point in a sufficiently small neighbourhood of $p$ would also be Levi corank one since this is an open condition. We say that $\Omega$ is a Levi corank one domain if $\partial \Omega$ is smooth, pseudoconvex, of finite type  and the Levi form of $ \partial \Omega $ has at least $ (n-2) $ positive eigenvalues everywhere on $ \partial \Omega$.

To describe the scaling method, let $ \Omega =\{ \rho(z, \overline{z}) <0 \} $ for some smooth defining function $ \rho $ and the $1$-type of $ p^0$ be $2m$. By relabelling the coordinates if necessary we assume that $(\pa \rho/\pa z_n)(p^0) \neq 0$. Then there exists a neighbourhood $U$ of $p^0$ such that $\vert \frac{\pa r}{\pa z_n}(p)\vert  \geq c$ for each $p \in U$ where $c$ is a constant. After a linear change of coordinates, we can find coordinates $z_1, \ldots, z_n$, and smooth functions $b_1, \ldots, b_{n-1}$ on $U$ such that
\[
L_n=\frac{\pa}{\pa z_n},\quad L_j=\frac{\pa}{\pa z_j}+b_j\frac{\pa}{\pa z_n}, \quad L_jr\equiv 0, \quad  b_j(p_0)=0 , \quad j=1, \ldots, n-1,
\]
which form a basis of $\mf{C}T^{1,0}(U)$ and satisfy
\[
\begin{bmatrix} \pa \ov \pa r (p_0) (L_i, \ov L_j)\end{bmatrix}_{2\leq i,j \leq n-1}=\mf{I}_{n-2}.
\]
Observe that $b_j(p^0)=0$ implies that the normal vector to $\pa D$ at $p^0$ is in the direction of $e_n=(0, \ldots,1)$. By \cite{Cho2}, shrinking $ U $ if necessary, for each $ \zeta \in U \cap \overline{\Omega} $, there is a global coordinate map $ \theta^{\zeta}$---which is a holomorphic polynomial automorphism of the form
\begin{equation}\label{cen-exp}
\theta^{\zeta}(z)=\left(z_1-\z_1, G_{\z}(\ti z - \ti \z)-Q_2(z_1-\z_1), \langle z-\z, \nu(\z)\rangle-Q_1({'}z-{'}\z)\right)
\end{equation}
where $G_{\z}\in GL_{n-2}(\mf{C})$, $\ti z=(z_2, \ldots, z_{n-1})$, $Q_1$ is a polynomial, $Q_2$ is a vector valued polynomial, and $\nu =(\pa \rho/\pa \ov z_1, \ldots, \pa \rho/\pa \ov z_n)$---so that the local defining function for $ \theta^{\zeta} (\Omega) $ near the origin is represented by 
\begin{multline}\label{nrmlfrm}
\left\{ z \in \mathbb{C}^n: \rho(\zeta) + 2 \Re z_n + \sum_{l = 2}^{2m}P_{l, \zeta}(z_1, \overline{z}_1) + \vert{z_2}\vert^2 + \ldots + \vert{z_{n-1}}\vert^2 + \right. \\
\left. \sum_{\al = 2}^{n - 1} \sum_{ \substack{j + k \le m\\
                                                          j, k > 0}} \Re \Big( \big(b_{jk}^{\al}(\z) z_1^j \ov z_1^k \big) z_{\al} \Big)+ R_{\z}(z) < 0 \right \},
\end{multline}
where 
\[
P_{l,\zeta}( z_1, \overline{z}_1) = \sum_{j + k = l} a^l_{jk}(\zeta) z_1^j \ov z_1^k,
\]
are real valued homogeneous polynomials of degree $l$ without harmonic terms and the error function $R_{\zeta}(z) \to 0$ as $z \to 0$ faster than one of the monomials of weight $1$. We will refer to $\theta^{\z}$ as the \textit{centering map} associated to $\z$.

Now choose $ \zeta^j \in \partial \Omega $ so that 
\[
\zeta^j = p^j + ('0, \epsilon_j), \; \epsilon_j > 0. 
\]
Let $ \theta^{\zeta^j} $ be the `centering maps' associated to $ \zeta^j \in \partial \Omega $. It follows from the explicit form \eqref{cen-exp} of the automorphisms $ \theta^{\zeta^j} $ that $ \theta^{\zeta^j}(\zeta^j) = ('0, 0) $ and 
\[
 \theta^{\zeta^j} (p^j) = 
\big('0, - \epsilon_j/ d_0(\zeta^j) \big) ,
\]
where $ d_0(\zeta^j ) = \Big( \partial \rho/\partial \overline{z}_n (\zeta^j) \Big)^{-1} \rightarrow  \Big( \partial \rho/\partial \overline{z}_n (p^0) \Big)^{-1}$ as $ j \rightarrow \infty $. Next, define 
\begin{alignat*}{3} 
\tau(\zeta^j, \epsilon_j) = 
\min_{ 2 \leq l \leq 2m}  \left( \frac{ \epsilon_j}{ \| P_{l, \zeta^j}(z_1, \overline{z}_1) \|} \right)^{1/l},
\end{alignat*}
where $ \| \cdot \| $ is the $ l^{\infty}$-norm on the finite-dimensional space of
polynomials on the complex plane with degree at most $ 2m $ as a finite sequence of coefficients. Denote by $ \Delta^{\epsilon_j}_{\zeta^j} : \mathbb{C}^n \rightarrow \mathbb{C}^n $ a dilation of coordinates given as follows
\[
\Delta^{\epsilon_j}_{\zeta^j} (z_1, z_2, \ldots, z_n ) = 
\left( \frac{z_1}{\tau(\zeta^j, \epsilon_j)}, \frac{z_2}{\epsilon_j^{1/2}}, \ldots, \frac{z_{n-1}} {\epsilon_j^{1/2}}, \frac{z_n} {\epsilon_j} \right).
\]
The scaling sequence is defined by setting $ A^j := \Delta^{\epsilon_j}_{\zeta^j} 
\circ \theta^{\zeta^j} $. Notice that
\[
q^j:= A^j(p^j) = \left('0, - 1/ d_0(\zeta^j) \right) \rightarrow q^0:=('0,-1/d_0(p^0)),
\]
and by \cite{TT}, the scaled domains $ \Omega_j = A^j(\Omega) $ converge in the Hausdorff sense to
\[
\Omega_{\infty} = \big\{ z \in \mathbb C^n : 2 \Re z_n + P_{\infty}(z_1, \ov z_1) + 
\vert z_2 \vert^2 + \ldots + \vert z_{n-1} \vert^2 < 0  \big\}
\]
where 
$ P_{\infty} $ is a subharmonic polynomial of degree at most $2m$ without harmonic terms. Further, it should be noted that if $ p^j $ approaches $ p^0 $ along the inner normal to $ \partial  \Omega $ at $ p^0 $, then the polynomial $P_{\infty}$ coincides with the polynomial of the same degree in the homogeneous Taylor expansion of the defining function $ \rho $ around the origin.


\subsection*{Scaling a strongly pseudoconvex polyhedral domain $ \Omega $:}
A bounded domain $ \Omega \subset \mathbb{C}^n $ is said to be a strongly pseudoconvex polyhedral domain with piecewise smooth boundary if 
there are $ C^2$-smooth real valued functions 
$ \rho_1, \ldots, \rho_k : \mathbb{C}^n \rightarrow \mathbb{R} $, $ k \geq 2 $ such that
\begin{itemize}
 \item $ \Omega = \big\{ z \in \mathbb{C}^n: \rho_1(z) < 0, \ldots, \rho_k(z) < 0 \big\} $,
 \item for $ \{ i_1, \ldots, i_l \} \subset \{1, \ldots, k\} $, the gradient vectors $ \nabla \rho_{i_1}(p), \ldots, \nabla \rho_{i_l}(p) $ are linearly independent 
 over $ \mathbb{C} $ for every point $ p $ such that $ \rho_{i_1}(p) = \ldots = 
 \rho_{i_l}(p) = 0 $, and
 \item $ \partial \Omega $ is strongly pseudoconvex at every smooth boundary point,
\end{itemize} 
where for each $ i = 1, \ldots, k $ and $ z \in \mathbb{C}^n $,
\[
 \nabla \rho_i (z) = 2 \left( \frac{\partial \rho_i}{\partial \bar{z}_1}(z), \ldots, \frac{\partial \rho_i}{\partial \bar{z}_n}(z) \right).
\]
Since the intersection of finitely many domains of holomorphy is a domain of holomorphy, it follows that the polyhedral domain $ \Omega $ is pseudoconvex. The scaling method for a strongly pseudoconvex polyhedral domain $ \Omega \subset \mathbb{C}^2 $ was introduced in \cite{Kim-Yu} and the reader is referred to it for the following: There exists a sequence of biholomorphic maps $A^j: \Om \to \Om_j$ from $\Om$ onto the domains $\Om_j$ such that $\Om_j$ converges in the local Hausdorff sense to a domain $\Om_{\infty}$ which is one of $ \mathbb{B}^2 $, the bidisc $ {\Delta}^2 $, or a Siegel domain of second kind given by 
\begin{equation} \label{Sie}
\left\{ (z_1, z_2) \in \mathbb{C}^2: \Im z_1 + 1 > \frac{ Q_1(z_2)}{m^2}, \Im z_2 > -1 \right\}, 
\end{equation}
where $ m > 0 $ and $ Q_1 $ is a strictly subharmonic polynomial of degree $ 2 $, and $q^j:=A^j(p^j) \to q^0 \in \Om_{\infty}$. In particular, note that $ \Omega_{\infty} $ is taut.

\subsection*{Scaling bounded convex domains:}
Firstly, suppose $\pa \Om$ is smooth. Then, as mentioned before, there are two cases to be considered:
\begin{enumerate}
 \item[(I)] $ \partial \Omega $ is of finite $1$-type near $ p^0$, or
 \item[(II)] $ \partial \Omega $ is of infinite $1$-type near $ p^0$.
\end{enumerate}
\noindent \textbf{Case (I):} Assume that $ \Omega $ is given by a smooth defining
function $ \rho $, $ p^0 $ is the origin and has type $2m$, and that $ \nabla \rho \left(('0,0)\right) = ('0,1) $. Consider the domain
\[
\Omega_{q,\epsilon} = \big\{ z : \rho(z) < \rho(q) + \epsilon \big \},
\]
for $ q \in \Omega$ sufficiently close to $ \partial \Omega $ and $ \epsilon > 0 $.  Choose $ s^n_{q,\epsilon} \in \partial \Omega_{q, \epsilon} $
where the distance of $q$ to $ \partial \Omega_{q, \epsilon} $ is
realized. Denote the complex line containing $q$ and $
s^n_{q,\epsilon} $ by $ L_n $. Let $ \tau_n(q, \epsilon) =
\big| q - s^n_{q,\epsilon} \big| $ and $ ( L_n)^{\bot} $ be the
orthogonal complement of the complex line $ L_n $ in $
\mathbf{C}^n$. Note that the distance
from $q$ to $ \partial \Omega_{q, \epsilon} $ along each complex line
in $ ( L_n)^{\bot} $ is uniformly bounded as $ \partial \Omega $ is of finite type. Let $ \tau_{n-1} (q,
\epsilon) $ be the largest such distance and $
s^{n-1}_{q,\epsilon} \in \partial \Omega_{q, \epsilon} $ be any point
such that $ \big | q - s^{n-1}_{q,\epsilon} \big| = \tau_{n-1}(q,
\epsilon)$. Denote the complex line containing $q$ and $
s^{n-1}_{q,\epsilon} $ by $ L_{n-1} $. Next, consider the orthogonal
complement of the $ \mathbb{C}$-subspace spanned by $ L_n $ and $
L_{n-1} $ and find the largest distance from $ q $ to $ \partial
\Omega_{q, \epsilon} $ therein. Choose $  s^{n-2}_{q,\epsilon}
\in \partial \Omega_{q, \epsilon} $ where this distance is realized.
Let $ \tau_{n-2} (q, \epsilon) = \big | q -  s^{n-2}_{q,\epsilon}
\big| $ and $ L_{n-2} $ denote the complex line containing $q$ and
$  s^{n-2}_{q,\epsilon} $. Repeating this process yields orthogonal lines $ L_n, L_{n-1}, \ldots, L_1 $. Let $ T^{q, \epsilon}
$ be the translation sending the point $ q $ to the origin and $ U^{q,
\epsilon} $ be a unitary transformation of $ \mathbb{C}^n $ sending $
L_i$ to the $ z_i $-axis and $  s^{i}_{q,\epsilon} - q $ to a
point on the $ \Re z_i $-axis. It follows by construction that
\begin{eqnarray*}
U^{q, \epsilon} \circ T^{q, \epsilon} (q) & = & 0 \\
\mbox{and} \qquad U^{q, \epsilon} \circ T^{q, \epsilon}
\big(s^i_{q, \epsilon} \big) & = & \big ( 0, \ldots, \tau_i(q, \epsilon), \ldots, 0)
\end{eqnarray*}
for all $ 1 \leq i \leq n$. 

For scaling $\Om$ along $\{p^j\}$, set $ \epsilon_j
= - \rho (p^j) $ and let $ \tau_1(p^j, \epsilon_j
), \ldots, \tau_n(p^j, \epsilon_j ) $ and $ s^{1,j}, \ldots, s^{n,j} $ be positive numbers associated with $ p^j $ and $ \epsilon_j $ as defined above. Define the
dilations
\[
\Lambda^{\epsilon_j}_{p^j} (z) = \left( \frac{
z_1}{\tau_1(p^j, \epsilon_j )}, \ldots, \frac{z_n}{\tau_n(p^j, \epsilon_j )} \right),
\]
the scaling sequence by setting $ A^j = \Lambda^{\epsilon_j}_{p^j}
 \circ U ^{p^j, \epsilon_j} \circ T^{p^j, \epsilon_j} $ and the scaled domains $ \Omega_j = A^j(\Omega) $. Note that $ \Omega_j $ is convex and $q^j:=A^j(p^j)=('0,0) \in \Omega_j $ for all $j$. Define $q^0:=({'0},0)$. By \cite{Gaussier-1997}, $ \Omega_j$ converges to
\[
\Omega_0 = \Big \{ ('z, z_n) \in \mathbb{C}^n : 
\Re \Big ( \displaystyle\sum_{k=1}^n b_k z_k  \Big) + P('z) < 1
\Big \}
\]
where $ b_k $ are complex numbers and $ P$ is a real convex
polynomial of degree less than or equal to $2m$. Furthermore, \cite{Mcneal-1992} shows that $ \Omega_0 $ is biholomorphically equivalent
to 
\[ 
\Omega_{\infty} = \big \{ ('z, z_n) \in
\mathbb{C}^n : 2 \Re z_n + P('z) < 0 \big \}. 
\]

\noindent \textbf{Case (II):} When $ p^0 \in \partial \Omega $ is a point of infinite type, it follows from Proposition 6.1 of \cite{Zimmer} that there exist complex affine transformations $ A^j $ of $ \mathbb{C}^n $ such that (after possibly passing to a subsequence) the domains $ \Omega_j = A^j (\Omega) $ converge to a convex domain $ \Omega_{\infty} $ in $ \mathbb{C}^n $. Also, $ q^j:=A^j(p^j) = q^0$ for some $ q^0 \in \Om_{\infty}$. Moreover, the limiting domain $ \Omega_{\infty} $ is complete hyperbolic and hence taut.

\subsection*{Non-smooth case}
We now deal with the case when smoothness properties of the boundary is not given but we are instead given that $\partial \Omega$ does not contain any nontrivial (complex) analytic varieties. We shall also assume that $M$ is complete hyperbolic in this sub-section. A good example of such a domain (with non-smooth boundary) to keep in mind here is the minimal ball which is defined as $\left\{z \in \mf{C}^n: N(z) < 1 \right\}$
and is the unit ball with respect to the norm
\[
N(z) = \frac{1}{\sqrt 2}\left( \sum_{j=1}^{n} \vert z_j \vert^2 + \big\vert \sum_{j=1}^{n} z_j^2 \big\vert \right)^{1/2}.
\]
More information on this can be found in \cite{HP}, \cite{OPY} for example.
Recalling that $p^j$ is the image of $z^0$ under the biholomorphisim $\psi^j: D_j \to \Om$, we shall modify the version of Frankel's scaling technique in Kim's article \cite{Kim1} (cf. also \cite{Kim2}) to 
apply it to our situation. We shall now do this in general (not just for the minimal ball), for a bounded convex domain $\Omega$ whose boundary does not contain nontrivial complex analytic varieties, as described in the following proposition.

\begin{prop} \label{scalprop}
With notations and assumptions as just-mentioned, we have: there exists a sequence $\{A_j\} \subset GL_n(\mathbb{C})$ such that
\begin{itemize}
\item[(i)] $\| A_j^{-1} \| \to 0$ as $j \to \infty$ and,
\item[(ii)] $\lim_{j \to \infty} A_j(\Omega - p^0) = \hat{\Omega}$ exists and is biholomorphic to $\Omega$, 
\end{itemize}
where the limit is taken in the sense of local Hausdorff convergence in $\mathbb{C}^n$ and where $\Omega-p^0$ denotes the translate
$\Omega - p^0 = \{ z - p^0 \; : \; z \in \Omega\}$.
\end{prop}

\begin{proof}
We contend that we may take $A_j$ to be the linear transformations $A_j: \mathbb{C}^n \to T_{z^0}M$ given by 
$A_j = [d \psi^j(z^0) ]^{-1}$. Let us break the proof into various steps in line with Kim's article \cite{Kim1} and define the affine-linear maps
$\tilde{A}_j : \mathbb{C}^n \to T_{z^0}M$ by 
\[
\tilde{A}_j(z) = \big[d \psi^j(z^0) \big]^{-1}\big(z- \psi^j(z^0)\big), 
\]
and $\varphi:  T_{z^0}M \times  T_{z^0}M \to  T_{z^0}M$, the mid-point map $\varphi(z, \zeta) = (z+\zeta)/2$.
We may identify the tangent space $T_{z^0}M$ to $M$ at the fixed base point $z^0$ with $\mathbb{C}^n$, even if in a non-canonical manner; this will be useful in some of the computations in the proofs and may well be done in what follows without further mention. Indeed, for the rest of this subsection we work with a fixed locally finite atlas $\mathcal{A}$ for $M$ and all local calculations are to be understood in the coordinates provided by charts of $\mathcal{A}$.
The first step for the proof of the proposition is then the following
\begin{lem} \label{omeglem}
The sequence $\omega_j : M_j \to T_{z^0}M$ defined by
\[
\omega_j(z) = \big[d \psi^j(z^0) \big]^{-1} \big( \psi^j(z) - \psi^j(z^0) \big),
\]
forms a normal family; moreover, any subsequential limit of $\{\omega_j\}$ is a biholomorphism of $M$ onto its image in $T_{z^0}M$.
\end{lem}

\noindent Note that $\omega_j(z^0)$ is the origin in $T_{z^0}M$ and to prove this lemma, it is enough to show that the sequence 
of `sup-norms' of their derivatives is bounded above
uniformly on compact subsets of $M$. To demonstrate this, first let $K_0$ (which we may well assume contains the point $z^0$) be a compact subset of $M$; then $K_0 \subset M_j$ for all large $j$ and $\omega_j$ alongwith all their derivatives are well-defined on $K_0$ for all $j \gg 1$ as well. As in \cite{Kim1} then,
choose another 
compact $K \subset M$ which contains $K_0$ in its interior, and  define the sequence of maps $F_m : K \times K \to M_m$ 
for all $m$ large, by
\[
F_m(z, \zeta) = \omega_m^{-1} \circ \varphi \circ (A_m \times A_m) \circ (\psi_m \times \psi_m).
\]
Note that all the $F_m$'s map the point $z^0$ to the origin and as they map into $\Omega$ which being a bounded convex domain is in particular complete hyperbolic (hence taut), $\{F_m\}$ is a normal family. Also note that we may rewrite the definition of the $F_m$'s to say that they 
satisfy the following relations with respect to $\omega_m$: 
\[
(\omega_m \circ F_m) (z, \zeta) = \frac{\omega_m(z) + \omega_m(\zeta)}{2}, 
\]
for each $m \in \mathbb{N}$. As in \cite{Kim1}, we can show after some computations that we have the following equations holding for the second-order derivatives of the $l$-th component function of $\omega_m$ for each $l=1,\ldots,n$, 
\begin{equation} \label{der-reln}
\frac{1}{2} \frac{\partial^2 \omega^l_m}{\partial z^k \partial z^i} (z) =
\sum_{j=1}^{n} \left( \frac{\partial^2 F^j_m}{\partial z^k \partial z^i} - \frac{\partial^2 F^j_m}{\partial \zeta^k \partial z^i} \right) (z,z)
\frac{\partial \omega^l_m}{\partial z^j}(z).
\end{equation}
From this, it can be deduced  that for all $j$, we have the following uniform bounds in terms of the sup-norms on $K$, for some positive constant $C_{K_0}$ (independent of $j$):
\begin{equation}\label{Kimeq}
\| d^2 \omega_j \|_{K_0} \leq C_{K_0} \| d \omega_j \|_{K_0} 
\end{equation}
where the left-hand-side denotes the max of the sup-norms of the various second-order partial derivatives of the component functions of the mapping $\omega_j$ on $K_0$ and likewise for the right; indeed for the first derivative appearing on the right, we shall work with the equivalent norm 
given by the maximum of the operator norms $\vert d \omega_j(z) \vert_{\rm op}$ for $z$ varying in the compact $K_0$. More importantly, while the constant $C_{K_0}$ may seem to depend on the  chart in which we obtained the previous equality, note that we may assume it to be independent of the  chart owing to the fact that each point of $M$ lies 
atmost in a finite number of charts of the atlas $\mathcal{A}$ and so, we may take $C_{K_0}$ to be the maximum of the constants arising when working with different charts containing $z$ as in (\ref{der-reln}); in essence, while the magnitudes of various partial derivatives 
depend on the coordinate chart, the fact that an inequality of the form (\ref{Kimeq}) holds is independent of the chart.
From this, we now further derive the following uniform bound on the first derivatives of the $\omega_j$'s to get the normality of $\{\omega_j\}$.

\begin{lem}
For some positive constant $C'_{K_0}$ depending only on $K_0$ and not on $j$, we have $\|d \omega_j \| \leq C'_{K_0}$
for all $j$.
\end{lem}

\begin{proof}
First note that it suffices to show this only for those compacts which are closures (in $M$) of relatively compact open connected subsets of  $M$ (because we may always exhaust $M$ by a sequence of 
relatively compact subdomains, by considering for instance, balls with respect to the Kobayashi distance centered at a fixed point and of a strictly increasing sequence of radii; the relative compactness of such balls is guaranteed by the complete hyperbolicity of $M$ and their connectedness by the fact that the Kobayashi distance is inner). We shall therefore henceforth assume that all compacts considered in the remainder of this proof are of this kind (i.e., is the closure of a relatively compact subdomain of $M$) and thereby in particular, connected; we shall also assume that such compacts contain the chosen base-point $z^0$. 
\smallskip\\
\noindent Now, let $U_z$ denote a relatively compact neighbourhood of the point $z \in K_0$ where we have a holomorphic chart. As the $U_z$'s form an 
open covering of $K_0$ as $z$ runs through $K_0$, we may extract a finite subcover whose union $U$ in particular then, is a relatively compact
subset of $M$ containing $K_0$; observe that $U$ is also connected. Let $K_1=\overline{U}$ and pick any point $z' \in K_0$.
Let $\gamma=\gamma^z(t)$ be any path within $U$ which joins $z^0$ to $z'$. Cover this path by finitely many relatively compact domains in $M$ say $N$-many, labelled $B_0, B_1, \ldots, B_{N-1}$ where each $B_k$ (for $0 \leq k \leq N-1$) is the biholomorphic image of a ball in $\mathbb{C}^n$
of some radius  
less than $1/4C_{K_1}$ where $C_{K_1}$ is the constant coming from (\ref{Kimeq}) applied to $K_1$; we may well assume that these domains are labelled
so that they form a chain $B_k \cap B_{k+1} \neq \phi$ for all $k$.\\
 
\noindent  We start with the first ball $B_0$ centered at the base-point $z^0$ where
 we have $d \omega_j (z^0) = {\rm identity}$ and consequently 
$\vert d \omega_j (z^0) \vert_{\rm op} = 1$ for {\it all} $j$; here $\vert \cdot \vert_{\rm op}$ denotes the operator norm
of the linear operator given by the derivative $D \omega_j(z^0)$. Then, identifying $B_0$ with its image in $\mathbb{C}^n$ for conciseness in writing by suppressing the chart maps involved, we have
\begin{align*}
\vert d \omega_j(z) \vert_{\rm op} &\leq \vert d \omega_j(z) - d\omega_j(z^0) \vert_{\rm op}  + \vert d \omega_j(z^0) \vert_{\rm op} \\
&\leq \| d^2\omega_j \|_{\overline{B}_0} \vert z - z^0 \vert + \vert d \omega_j(z^0) \vert_{\rm op} \\
& \leq C_{\overline{B}_0} \| d \omega_j \|_{\overline{B}_0} \vert z - z^0 \vert + \vert d \omega_j(z^0) \vert_{\rm op} 
\end{align*}
Then using the fact that $\vert d \omega_j(z^0) \vert_{\rm op}=1$ as already noted and more importantly, the fact that $z \in B_0$ -- thereby 
that $\vert z - z^0 \vert < 1/4C_{K_1} < 1/2C_{K_1}$ -- we may derive from the above inequality by taking supremum over $z \in \overline{B_0}$
and thereafter transposing a term on the right of the inequality to the left, that 
\[
\frac{\| d \omega_j(z) \|_{\overline{B}_0}}{2} \leq \vert d \omega_j(z^0) \vert_{\rm op} =1.
\]
That is, we have for all $j$ that  $\vert d \omega_j(z) \vert_{\rm op} \leq 2$ holds for all $z \in B_0$, in particular at a point 
$z^1 \in B_0 \cap B_1$. We run the above argument again, now for $z$ in the ball $B_1$ and $z^0$ replaced by $z^1$ i.e., apply 
(\ref{Kimeq}) as before to get
\begin{align*}
\vert d \omega_j(z) \vert_{\rm op} &\leq \vert d \omega_j(z) - d\omega_j(z^1) \vert_{\rm op}  + \vert d \omega_j(z^1) \vert_{\rm op} \\
&\leq  \| d^2\omega_j \|_{\overline{B}_1} \vert z - z^1 \vert + \vert d \omega_j(z^1) \vert_{\rm op} \\
& \leq C_{K_1} \| d \omega_j \|_{\overline{B}_1} \vert z - z^1 \vert + \vert d \omega_j(z^1) \vert_{\rm op} 
\end{align*}
Now use the fact that $z,z^1$ both lie in $B_1$ means that the distance $\vert z - z^1 \vert$ (as measured in the local holomorphic chart) between them is at-most $2/4C_{K_1}$ to get
\[
\frac{\| d \omega_j(z) \|_{\overline{B}_1}}{2} \leq \vert d \omega_j(z^0) \vert_{\rm op} =2.
\]
That is, we have for all $j$ that $\vert d \omega_j(z) \vert_{\rm op} \leq 2^2$ holds for all $z \in B_1$. Proceeding inductively and running the above (only finitely many times) argument for each of the balls $B_k$, we conclude for all $z$ in their union that for all $j$, we have
 $\vert d \omega_j(z) \vert_{\rm op} \leq 2^N$. In particular, we have this holding for $z'$, the terminal point of $\gamma$. But then as $z'$ was an arbitrarily chosen point of $\gamma$, we get that for all $j$, we have
\[
\| d \omega_j \|_{K_0} \leq 2^N.
\]
Note that $N$ depends on $K_0$: it may be thought of as the minimum number of balls of radius $1/4C_{K_1}$ required to cover $K_1$(which was essentially a thickening of $K_0$); however, it is independent of $j$ and so we are done.
\end{proof}


Hence, we conclude that $\omega_j$ forms a normal family of holomorphic mappings. To complete the proof of Lemma \ref{omeglem}, we still need to show that every subsequential limit $\omega$ of the $\omega_j$'s gives a holomorphic embedding 
of $\Omega$ (again in $T_{z^0}M \simeq \mathbb{C}^n$). This can be seen as follows. Firstly, assume after passing to a subsequence that the 
$\omega_j$'s converge uniformly on compacts to $\omega$. Note that as the $\omega_j$'s are all holomorphic embeddings of 
the subdomains $D_j$ of $M$, their Jacobian determinants are all nowhere vanishing on $\Omega$ to which when we apply Hurwitz's theorem,
we deduce (by noting also that $\omega_j(z^0)=0$ and $d \omega_j (z^0) ={\rm identity}$) that the Jacobian determinant of $\omega$ must be nowhere vanishing as well, thereby (by the inverse function theorem) that $\omega$ is an open map. However, the uniform convergence of the $\omega_j$'s on compact subsets of $M$ actually implies that $\omega$ is (globally) one-to-one i.e, $\omega$ is a 
biholomorphism of $M$ onto $\omega(M)\subset T_{z^0}M$. Moreover, it can be shown that:

\begin{lem} \label{eigen}
Every eigenvalue of $d \psi^j(z^0)$ and $\| d \psi^j(z^0) \|$ tends to $0$ as $j \to \infty$.
\end{lem}

\begin{proof}
We intend to prove this by contradiction and since the norm of a linear operator dominates the magnitude of all the eigenvalues, we may suppose (to obtain a contradiction) that there exists a vector $v \in T_{z^0}M$ such that for all $j \gg 1$, we have the lower bound:
\begin{equation}\label{psijlowbd}
\lim_{j \to \infty} \| d \psi^j(z^0)v \| \geq \epsilon>0.
\end{equation}
Let $f: \Delta \to M$ be any (non-constant) analytic disc passing through $z^0$ and $\Delta_0$ a small disc about $0 \in \Delta$ such that 
the $\psi^j$'s for all $j \gg 1$ are all well-defined on the image $f(\Delta_0)$. The sequence of analytic discs $\psi^j \circ f$ is then
well-defined for all $j \gg 1$ on $\Delta_0$ and being mapping into the bounded domain $\Omega$, admits a subsequence, that converges uniformly on compact subsets of $\De_0$ to a holomorphic map $\Phi:\De_0 \to \ov \Om$. Since $\Phi(0)=p^0 \in \partial \Omega$ and $\Om$ is taut, $\Phi$ must map $\Delta_0$ entirely into $\partial \Omega$, which however we know does not admit any non-constant analytic disc. Therefore, $\Phi$ must be constant. But then, 
this contradicts the fact $\|\Phi'(0)\| \geq \epsilon>0$, as follows from (\ref{psijlowbd}), and finishes the proof.
\end{proof}

As the next step, we show 

\begin{lem} \label{omegcvgce}
After possibly passing to a subsequence of $\{\omega_j\}$ which we may well assume (by the foregoing) to be convergent to a holomorphic embedding 
$\omega:M \to T_{z^0}M$, we have that the sequence of domains in $T_{z^0}M$ 
given by their images, namely $\{\omega_j(M_j)\}$, converges as $j \to \infty$ in the local Hausdorff sense to $\omega(M)$; also, 
$\omega(M)$ is convex.
\end{lem} 

\begin{proof}
Let us begin by remarking that to show containment of compacts as in the definition of local Hausdroff convergence, it suffices to restrict attention to those $K$ which are closures of relatively compact domains in $M$,
indeed a sequence of such compacts which form an exhaustion of $M$. Getting to the verification of the asserted convergence of the images 
$\omega_j(M_j)$ more precisely now, fix a pair of such compacts $K_1, K_2  \subset \omega(M)$ with $K_1$ contained in the interior $K_2^0$  of 
$K_2^0$. We need to show that $K_1 \subset \omega_j(M_j)$ for all $j \gg 1$. For this, first write $K_1=\omega(S_1), K_2= \omega(S_2)$ for 
some $S_1,S_2 \subset M$. As $\omega$ is a holomorphic embedding, it follows that $S_1,S_2$ are also compact and $S_1$ is contained in the 
the domain $S_2^0$, the interior of
$S_2$ (it may also be noted that $S_1,S_2 \subset M_j$ for all $j \gg 1$). 
As $\omega_j$'s are all holomorphic embeddings in particular open maps, we observe that we may follow the arguments of proposition 5 of chapter-5 of \cite{Nar} applied to the mappings $\omega_j$ restricted to $S_2^0$, to conclude that $\omega_j (M_j) \ni \omega(S_1) = K_1$ for all 
$j \gg 1$, finishing one half of the verification. To complete the proof of the asserted local Hausdorff convergence then, we
now conversely fix a compact $K \subset T_{z_0}M$ with the property that $K \subset \omega_j(M_j)$ holds for all $j \gg 1$; indeed, observe that for this part of the argument, we may assume as well that $K$ is contained in the interior $U$ of the intersection of all the $\omega_j(M_j)$'s.
We shall show now that
$K \subset \omega(M)$. To this end, write
$K=\omega_j(S_j)$ for $S_j \subset M_j$ (it may be noted that as each $\omega_j$ is an embedding, the $S_j$'s are compact, though this is not needed here). Fix any $k \in K$; note then that we have for all $j \gg 1$, points $s_j \in S_j$ with $\omega_j(s_j) = k$. We observe that 
the complete hyperbolicity of $M$ forces $\{s_j\}$ to be compactly contained in $M$; indeed,
\[
d_M(s_j, z^0) \leq d_{M_j}(s_j, z^0) \leq d_{\omega_j(S_j)}(\omega_j(s_j),\omega_j(z^0)) 
= d_{\omega_j(M_j)}(k, 0) \leq d_U(k,0).
\]
Note that by in the above chain of inequalities on the Kobayashi distances of the various associated domains, the right most member is independent of $j$, showing that the $s_j$ remains within a fixed distance of the base point $z^0$ in $M$. Thus by completeness of $M$, after passing to a subsequence, we may assume $s_j \to s_0 \in M$ thereby, $\om(s_0)= \lim \om_j(s_j) =k$ and hence, $k \in \om(M)$. As $k$ was an arbitrary point 
of $K$, this finishes the proof of $K \subset \omega(M)$ and therewith of the asserted local Hausdorff convergence.
\smallskip\\
For the remaining assertions about convexity, note that from the definition of the $\omega_j$'s, particularly from $\psi^j(M_j) = \Omega$ for all $j$,  we have
\begin{equation} \label{omegjimages}
\omega_j(M_j) = \big[d \psi^j(z^0)\big]^{-1}(\Omega - p^j)
\end{equation}
where as we know $p^j \to p \in \partial \Omega$. While the above equation makes the convexity of the $\omega_j(M_j)$'s -- which are affine-linear images of $\Omega$ by (\ref{omegjimages}) -- apparent, their local Hausdorff convergence to $\omega(M)$ established above, then 
ensures the convexity of $\omega(M)$ as well.%
\end{proof}

Finally, we show

\begin{lem} \label{sigmjnrml}
The sequence of holomorphic mappings $\sigma_j : M_j \to T_{z^0}M$ defined by 
\begin{equation} \label{sigpsireln}
\sigma_j(z) = \big[d\psi^j(z^0)\big]^{-1}\big(\psi^j(z) - p^0\big)
\end{equation}
is a normal family. Moreover, any subsequential limit is a holomorphic embedding of $M$ into $T_{z^0}M$.
\end{lem}

\begin{proof}
For the proof of this lemma, first note the relation between the maps $\sigma_j$ defined here and the $\omega_j$'s that we had earlier:
\[
\sigma_j(z) = \omega_j(z) + \big[d\psi^j(z^0)\big]^{-1}(p^j-p^0)
\]
which means in particular that for each $j$, the difference between $\sigma_j(z)$ and $\omega_j(z)$ is independent of $z$. This yields the normality asserted in the lemma as soon as we can establish that the 
sequence of points $[d\psi^j(z^0)]^{-1}(p^0-p^j)$ (where as we know, $p^j = \psi^j(z^0)$) in $T_{z^0}M$, is a bounded sequence. To argue this by contradiction, we suppose that this sequence is unbounded and after passing to a subsequence (which we shall tacitly keep doing wherever required in the remainder of this proof) assume that 
\[
\lim_{j \to \infty} \Big\| \big[d\psi^j(z^0)\big]^{-1}(p^0 - p^j) \Big\| = \infty.
\]
We claim that a similar statement (more precisely (\ref{bdyinfty}) below) holds for all other points $t \in \partial \Omega$ as well. To see this claim, first recall Lemma \ref{eigen} according to which, the minimum modulus of the eigenvalues of the linear operators $A_j=[d\psi^j(z^0)]^{-1}$ diverges to infinity. This means that the $A_j$'s diverges uniformly on compacts of $\mathbb{C}^n \setminus \{0\}$; considering in particular, the sequence of points $t-p^j$ together with its limit $t-p^0$, which are all contained (compactly) within $\mathbb{C}^n \setminus \{0\}$,  we deduce that: for each point $t \in \partial \Omega$ for $t \neq p$ also, we have that 
\begin{equation} \label{bdyinfty}
\lim_{j \to \infty} \Big\| \big[d\psi^j(z^0)\big]^{-1}(t - p^j) \Big\| = \infty.
\end{equation}
Note that as $t$ varies in the compact set $\partial \Omega$, the points appearing in the above namely, $\tilde{A}_j(t)=[\psi^j(z^0)]^{-1}(t - p^j)$ for each fixed $j$, runs through the boundary of the domain $[d\psi^j(z^0)]^{-1}(\Omega - p^j)$ i.e., the image of $\partial \Omega$ under the map $\tilde{A}_j$. But then equation (\ref{omegjimages}) gives the following relation of this to the boundary of 
the images of $\omega_j$'s namely,
\[
\partial \omega_j(M_j) =  \big[d\psi^j(z^0)\big]^{-1}(\partial \Omega - p^j).
\]
Therefore, by what we noted above from (\ref{bdyinfty}), $\omega_j(M_j)$ converges in the local Hausdorff sense, to the entire $\mathbb{C}^n$ i.e.,
$\omega(M) = T_{z^0}M \simeq \mathbb{C}^n$ which however is impossible as $M$ is hyperbolic and $\omega$ was already verified to be a holomorphic embedding. 
Thus we conclude the existence of a $t^0 \in \partial \Omega$ such that for some positive constant $C$ we have
\begin{equation} \label{tildAbd}
\big\| \tilde{A}_j (t^0) \big\| \leq C
\end{equation}
for all $j$; here, as we know $\tilde{A}_j$ are the affine-maps given by $\tilde{A}_j(z) = A_j(z-p^j)=[d \psi^j (z^0)]^{-1}(z-p^j)$. As the $A_j$'s diverge uniformly on compacts of $\mathbb{C}^n \setminus \{0\}$, it follows that the only point $t^0 \in \partial \Omega$ which can satisfy 
(\ref{tildAbd}) is $t^0=p^0$. Therefore, $\tilde{A}_j(p^0) = [d\psi^j(z^0)]^{-1}(p^0-p^j)$ is a bounded sequence, which was what was pending to be established  to obtain the normality of the family $\sigma_j$'s as noted at the outset. Recalling (\ref{sigpsireln}), we also get that every subsequential limit of the $\sigma_j$'s is a biholomorphism of $M$ onto $\sigma(M)$, since the same is true of the $\omega_j$'s by Lemma \ref{omeglem}. 
\end{proof}

\begin{lem}\label{A_j-sig_j}
We have for all $j$ that 
\[
A_j(\Omega-p^0) = \sigma_j(M_j)
\]
\end{lem}

\begin{proof}
We unravel the set on the left-hand-side as follows by using the following definition of the difference between a 
pair of subsets $A,B$ in $\mathbb{C}^n$: $A - B =\{ a -b \; : \; a \in A,\; b \in B \}$. Then,
\begin{align*}
A_j(\Omega-p^0) &= \big[d \psi^j(z^0)\big]^{-1} (\Omega - p^0) \\
&= \big[d \psi^j(z^0)\big]^{-1} \big( \psi^j(M_j) -p^0 \big) \\
&= \big[d \psi^j(z^0)\big]^{-1} \Big( \big(\psi^j(M_j) - \psi^j(z^0)\big) + \big(\psi^j(z^0) - p^0\big) \Big) \\
&=  \omega_j(M_j) - (\omega_j - \sigma_j) (M_j) \\
&= \sigma_j(M_j)
\end{align*}
The penultimate equality follows by noting that $[d \psi^j(z^0)]^{-1}  (\psi^j(M_j) - \psi^j(z^0))=\omega_j(M_j)$, by using the linearity of $ [d \psi^j(z^0)]^{-1}$ and, the fact that $(\omega_j - \sigma_j)(z)$ is actually independent of $z$ with the images of these constant maps forming the sequence $[d\psi^j(z^0)]^{-1}(\psi^j(z^0) - p^0)$.
\end{proof}
This essentially finishes the proof of Proposition \ref{scalprop}. The assertion in the proposition that the scaled domains $A_j(\Omega-p^0)$ converges in the local Hausdorff sense to a domain $\hat{\Omega} \subset T_{z^0}M \simeq  \mathbb{C}^n$ now follows from the last lemma combined with the fact about normality of the $\sigma_j$'s as in lemma \ref{omegcvgce}; that $\hat{\Omega}$ is biholomorphic to $\Omega$ itself
follows from the fact that the limits of $\{ \sigma_j \}$ are holomorphic embeddings.
\end{proof}
\noindent We define the scaling maps by $A^jz:=A_j(z-p^0)$, and set $\Om_j:=A^j(\Om)$ and $\Om_{\infty}:=\hat{\Om}$. Note that $\Om_{\infty}$, being biholomorphic to the bounded convex domain $\Om$, is complete hyperbolic and hence taut. Also, by Lemma~\ref{sigmjnrml},
\[
q^j:=A^j(p^j)=A_j(p^j-p^0)=\big[d\psi^j(z^0)\big]^{-1} \big(\psi^j(z^0)-p^0\big)=\sigma_j(z^0),
\]
converges after passing to a subsequence, to some $q^0 \in \lim \sigma_j(M_j)$. By Lemma~\ref{A_j-sig_j}, $ \lim \sigma_j(M_j)=\lim A_j(\Om-p^0)=\hat{\Om}=\Om_{\infty}$ and so $q^0 \in \Om_{\infty}$, as well.
\medskip\\
This finishes the description of scalings for various classes of domains and we are now ready to provide
\begin{proof}[Proof of Proposition~\ref{k-stability}]
The normality of the mappings $\ti \psi^j$ follows from the proof of Theorem~3.11 of \cite{TT} (when $\Om$ is a Levi corank one domain), Lemma~3.1 of \cite{Gaussier-1997} (when $\Om$ is a convex finite type domain), Proposition 4.2 of \cite{Zimmer} (for convex infinite type domains and the minimal ball). Moreover,
\begin{equation} \label{E7}                                                                                                                                                                                                                                                      \tilde{\psi^j} (z^0) = A^j \circ \psi^j (z^0) = A^j(p^j) \rightarrow q^0,
\end{equation}
by construction, and so the tautness of $ \Omega_{\infty} $ forces that the uniform limit $ \tilde{\psi} $ is a holomorphic mapping from $ M $ into $ \Omega_{\infty} $.

We now prove \eqref{E4} and one of the key ingredient is the stability of the infinitesimal Kobayashi metric, i.e.,
\begin{equation} \label{E5}
F_{\Omega_j} (\cdot, \cdot) \rightarrow F_{\Omega_{\infty}} (\cdot, \cdot)
\end{equation} 
uniformly on compact sets of $ \Omega_{\infty} \times \mathbb{C}^n $. The above statement is established by examining the limits of holomorphic mappings $ f^j : \Delta \rightarrow \Omega_j $ that almost realize $ F_{\Omega_j} (\cdot, \cdot) $. It is known (refer Theorem 3.11 of \cite{TT} for Levi corank one domains $ \Omega $, \cite{Gaussier-1997} and Proposition 4.2 of \cite{Zimmer} for convex domains $ \Omega $, and that $\Omega_j \subset 2 \Omega_{\infty} $ for all large $ j $ for strongly pseudoconvex polyhedral domains $\Om$) that $ \{f^j\} $ is normal on $ \Delta $ and, after passing to a subsequence if necessary, converges to a holomorphic mapping $ f: \Delta \rightarrow \Omega_{\infty} $. Then the uniform limit $ f $ is a competitor for $ F_{\Omega_{\infty}} (\cdot, \cdot) $.

Fix $ \epsilon > 0 $ and let $ \gamma : [0,1] \rightarrow \Omega_{\infty} $ be a piecewise $C^1$-smooth path in $ \Omega_{\infty} $ such that $ \gamma(0)= \tilde{\psi} (z^1) $, $ \gamma(1)= \tilde{\psi} (z^2) $ and
\[
 \int_0^1 F_{\Omega_{\infty}} \left( \gamma(t), \dot{\gamma}(t) \right) \leq d_{\Omega_{\infty}} \big( \tilde{\psi} (z^1),\tilde{\psi} (z^1) \big) + \epsilon/2.
\]
Notice that the trace of $ \gamma $ is compactly contained in $ \Omega_{\infty} $ and hence, the trace of $ \gamma $ is contained uniformly relatively compactly in $ \Omega_j $ for all large $ j $. Moreover, it follows from (\ref{E5}) that
\begin{alignat*}{3}
\int_0^1 F_{\Omega_j} \left( \gamma(t), \dot{\gamma}(t)\right) \leq \int_0^1 F_{\Omega_{\infty}} \left( \gamma(t), \dot{\gamma}(t)\right) + \epsilon/2 \leq d_{\Omega_{\infty}} \big(\tilde{\psi} (z^1), \tilde{\psi} (z^2) \big) + \epsilon/2,
\end{alignat*}
and consequently that
\begin{alignat*}{3}
d_{\Omega_j} \left( \tilde{\psi} (z^1), \tilde{\psi} (z^2) \right) \leq \int_0^1 F_{\Omega_j} \left( \gamma(t), \dot{\gamma}(t)\right) \leq d_{\Omega_{\infty}} \big(\tilde{\psi} (z^1), \tilde{\psi} (z^2) \big) + \epsilon/2,
\end{alignat*}
which, in turn, implies that
\[
\limsup_{j \rightarrow \infty} d_{\Omega_j} \big( \tilde{\psi} (z^1), \tilde{\psi} (z^2) \big) \leq d_{\Omega_{\infty}} \big( \tilde{\psi} (z^1), \tilde{\psi} (z^1) \big)
\]
as required.

Note that for convex infinite type domains and the minimal ball, Theorem 4.1 of \cite{Zimmer} guarantees the stability of the integrated Kobayashi distance under scaling, thereby rendering (\ref{E4}).
\end{proof}

\subsection{Step II}
We establish that $\ti \psi: M \to \Om_{\infty}$, where $\ti \psi$ is a limit of $\ti \psi^j=A^j\circ \psi^j$ given by Proposition~\ref{k-stability}, is a biholomorphisim. The most natural candidate for the inverse of $\ti \psi$ is a limit, if exists, of the backward scaling sequence
\begin{equation} \label{E6}
\tilde{\phi^j} := \left( \tilde{\psi^j} \right)^{-1} = \left( A^j \circ \psi^j \right)^{-1} : \Omega_j \rightarrow M_j \subset M.
\end{equation}
However, it is not \textit{a priori} evident that this sequence has a limit as the target domain $D$ need not be taut and this is the principal difficulty in proving that $\ti \psi$ is invertible. So first we establish that $\ti \psi $ is injective which will enable as to identify $M$ with $\ti\psi(M)$ which is a subdomain of the taut domain $\Om_{\infty}$ and eventually to get hold on a limit of the above sequence.

To see that $\ti \psi$ is injective, let $ z^1 $ and $ z^2 $ be any two points in $ M $ and $ d_M( \cdot, \cdot) $ denote the integrated Kobayashi distance on $ M $. Then for each $ j $,
\begin{equation} \label{E1}
 d_M (z^1, z^2) =  d_M \left( \tilde{\phi^j} \circ \tilde{\psi^j} (z^1), \tilde{\phi^j} \circ \tilde{\psi^j} (z^2) \right).
\end{equation}
The distance decreasing property of the holomorphic mappings implies that 
\begin{equation} \label{E2}
 d_M \left( \tilde{\phi^j} \circ \tilde{\psi^j} (z^1), \tilde{\phi^j} \circ \tilde{\psi^j} (z^2) \right)
 \leq d_{\Omega_j} \left( \tilde{\psi^j} (z^1), \tilde{\psi^j} (z^2) \right),
\end{equation}
and the triangle inequality yields
\begin{equation} \label{E3}
 d_{\Omega_j} \left( \tilde{\psi^j} (z^1), \tilde{\psi^j} (z^2) \right) \leq d_{\Omega_j} \left( \tilde{\psi^j} (z^1), \tilde{\psi} (z^1) \right) + d_{\Omega_j} \left( \tilde{\psi} (z^1), \tilde{\psi} (z^2) \right)+ d_{\Omega_j} \left( \tilde{\psi} (z^2), \tilde{\psi^j} (z^2) \right).
\end{equation}
Note that the terms $ d_{\Omega_j} \left( \tilde{\psi^j} (z^1), \tilde{\psi} (z^1) \right)  $ and $ d_{\Omega_j} \left( \tilde{\psi} (z^2), \tilde{\psi^j} (z^2) \right) $ converge to $ 0 $. Indeed, observe that $ \tilde{\psi^j} (z^1) \rightarrow \tilde{\psi} (z^1)$ and the domains $ \Omega_j $ converge to $ \Omega_{\infty} $. As a consequence, there is a small Euclidean ball $ B \left( \tilde{\psi} (z^1), r \right) $ centered at $ \tilde{\psi} (z^1) $ which contains $ \tilde{\psi^j} (z^1) $ for all $ j $ large and which is contained in $ \Omega_j $ for all $ j $ large, where $ r > 0 $ is independent of $ j $. It follows that 
\[
 d_{\Omega_j} \left( \tilde{\psi^j} (z^1), \tilde{\psi} (z^1) \right) \leq C | \tilde{\psi^j} (z^1)- \tilde{\psi} (z^1) |, 
\]
where $ C > 0 $ is independent of $j $ and consequently that 
\[
d_{\Omega_j} \left( \tilde{\psi^j} (z^1), \tilde{\psi} (z^1) \right) \rightarrow 0 
\]
as $ j \rightarrow \infty $. A similar argument shows that 
\[
d_{\Omega_j} \left( \tilde{\psi} (z^2), \tilde{\psi^j} (z^2) \right) \rightarrow 0 
\]
as $ j \rightarrow \infty $. Finally, by Proposition~\ref{k-stability}, the middle term $ d_{\Omega_j} \left( \tilde{\psi} (z^1), \tilde{\psi} (z^2) \right) $ on the right hand side of the inequality (\ref{E3}) is dominated by $d_{\Om_{\infty}}(\ti \psi(z_1), \ti \psi(z_2))$. Therefore, combining the inequalities (\ref{E1}), (\ref{E2}) and (\ref{E3}) and letting $ j \rightarrow \infty $ gives 
\begin{equation*} 
 d_M (z^1, z^2) \leq d_{\Omega_{\infty}} \left( \tilde{\psi} (z^1), \tilde{\psi} (z^2) \right). 
\end{equation*}
The hyperbolicity of $ M $ guarantees that $ z^1 = z^2 $ whenever $ \tilde{\psi} (z^1) = \tilde{\psi} (z^2) $. Hence, $ \tilde{\psi} : M \rightarrow \Omega_{\infty} $ is injective.

We now prove that $ \tilde{\psi} $ is surjective by verifying that the inverse scaling sequence $ \tilde{\phi^j} $ posseses a limit that serves as an inverse to $\tilde{\psi}$. Since $M$ is biholomorphic to $ \tilde{\psi}(M) \subset \Omega_{\infty}$ and $\Om_{\infty}$ is taut, we can consider $ M$ as a submanifold of some taut manifold $ M'$ and therefore,
\[
\tilde{\phi^j} = \left( \tilde{\psi^j} \right)^{-1}: \Omega_j \rightarrow M_j \subset M \subset M'
 \]
admits a subsequence that either converges uniformly on compact subsets of $\Om_{\infty}$ to a holomorphic mapping $ \tilde{\phi} : \Omega_{\infty} \rightarrow M' $ or diverges uniformly on compact subsets of $\Om_{\infty}$. Since $\tilde{\phi^j}(q^j)=z^0$ and $q^j \to q^0$, the latter case cannot occur. Keeping the same notation $ \{ \tilde{\phi^j} \} $ for this convergent subsequence, we have $\ti \phi: \Om_{\infty} \to \ov M\subset M'$. We claim that $\ti \phi: \Om_{\infty} \to M$ and to prove our claim it is enough to show that $\ti \phi$ is an open map. Being the limit of a sequence of biholomorphic maps we know that by Hurwitz's Theorem either the Jacobian of $ \tilde{\phi} $ is never zero or identically zero on $ \Omega_{\infty} $. Thus the claim would be verified once we exclude the second possibility and for this all that is required is to produce an open set in $\ti \phi(\Om_{\infty})$. In fact, we will now show that if $G\subset\subset M_1$ is an open set then $G \subset \ti\phi(\Om_{\infty})$. Indeed, let $z \in G$. Then $z \in D^j$ for all $j$ and let $w^j=\ti\psi^j(z) \in \Om^j$. Then $w^j \to w^0=\ti \psi(z) \in \ti \psi(G)$. Recalling that $\ti \psi : M \to \Om_{\infty}$ is an injective holomorphic map, first we have $\ti \psi(G)$ is an open subset of $\Om_{\infty}$ so that $w^j \in \ti \psi(G)$ for all large $j$, and second we also have $\ti \psi(G) \subset \subset \Om_{\infty}$ so that $\ti \phi^j$ converges uniformly on $\ti \psi(G)$ to $\ti \phi$. Therefore,
\[
z=\ti \phi^j(w^j) \to \ti \phi (w^0) \in \ti \phi(\Om_{\infty}).
\]
This proves that $G \subset \ti \phi(\Om_{\infty})$ and hence our claim.
Now it is immediate that the sequence $ \{ \tilde{\psi^j} \circ \tilde{\phi^j} \} $ of the identity mappings converges to $ \tilde{\psi} \circ \tilde{\phi} $ on $ \Omega_{\infty} $, i.e., 
\[
  \tilde{\psi} \circ \tilde{\phi} (z) = \lim_{j \rightarrow \infty} \tilde{\psi^j} \circ \tilde{\phi^j} (z) = z,
\]
for $ z \in \Omega_{\infty} $. In particular, $ \Omega_{\infty} \subset \tilde{\psi}(M) $ and this proves the surjectivity of $ \tilde{\psi} $. Thus $M $ is biholomorphic to $ \Om_{\infty}$ which completes the proof of Theorem~\ref{T1} (i)-(iv).

\subsection{Proof of Theorem~\ref{T1} (v): When $ \Omega $ has generic piecewise smooth Levi-flat boundary}
A bounded domain $ \Omega \subset \mathbb{C}^n $ is said to have generic piecewise smooth Levi-flat boundary if there exists a neighbourhood $ U $ of $ \overline{\Omega } $ and 
$ C^{\infty}$-smooth real valued functions 
$ \rho_1, \ldots, \rho_k : U \rightarrow \mathbb{R} $, $ k \geq 2 $ such that
\begin{itemize}
 \item $ \Omega = \big\{ z \in U : \rho_1(z) < 0, \ldots, \rho_k(z) < 0 \big\} $,
 \item for $ \{ i_1, \ldots, i_l \} \subset \{1, \ldots, k\} $, $ d \rho_{i_1} \wedge \cdots \wedge d \rho_{i_l} \neq 0 $ for every point $ p $ such that $ \rho_{i_1}(p) = \ldots = 
 \rho_{i_l}(p) = 0 $, 
 \item for $ \{ i_1, \ldots, i_l \} \subset \{1, \ldots, k\} $, $ \partial \rho_{i_1} \wedge \cdots \wedge \partial \rho_{i_l} \neq 0 $ for every point $ p $ such that $ \rho_{i_1}(p) = \ldots = \rho_{i_l}(p) = 0 $, and,
 \item $ \partial \Omega $ is Levi-flat at every smooth boundary point.
\end{itemize}

The main reference and motivation comes from \cite{Fu-Wong}, wherein the authors use the $ c/k $-invariant (with respect to $\De^2$) to show that if $ \Omega \subset  \mathbb{C}^2 $ is a simply connected domain having generic piecewise smooth Levi-flat boundary and non-compact automorphism group, then the $ c/k $-invariant along the non-compact orbit approaches $ 1 $ and consequently $ \Omega $ is biholomorphic to $ \Delta^2 $. For a bounded convex domain $ \Omega \subset \mathbb{C}^n $ with piecewise $ C^{\infty}$-smooth Levi-flat boundary, scaling method is known thanks to K.T. Kim (see \cite{Kim}), and the ideas therein can be adapted for the Union problem for such domains. However, without the convexity assumption, i.e., for a bounded domain $ \Omega $ in $ \mathbb{C}^n $ with generic piecewise smooth Levi-flat boundary, scalings do not seem plausible and therefore S. Fu and B. Wong appeal to the $ c/k $-invariant.

To recall the definition of the $c/k$ invariant on a complex manifold $X$, let $p \in X$ and $z_1, \ldots, z_n$ be local holomorphic coordinates centred at $p$. Set
\begin{equation}\label{v-form-coeff}
\begin{aligned}
 C_X(p) & = \sup \big\{ | \det f'(p)|^2: f \in \mathcal{O} (X, \mf{B}^n), f(p) = 0 \big\},\\
 K_X(p) & = \inf \big\{ | \det g'(0)|^{-2}: f \in \mathcal{O} (\mf{B}^n, X), g(0) = p \big\},
\end{aligned}
\end{equation}
where $f'(p)$ and $g'(0)$ are the Jacobian matrices computed in the coordinates $z$ in $X$ and the standard coordinates in $\mf{C}^n$. It is evident that
\begin{equation}\label{v-form}
\begin{aligned}
c_X(p) & =C_X(p) \left(\frac{i}{2}\right)^n dz_1\wedge d\ov z_1 \wedge \cdots \wedge dz_n \wedge d\ov z_n\\
k_X(p)&=K_X(p) \left(\frac{i}{2}\right)^n dz_1\wedge d\ov z_1 \wedge \cdots \wedge dz_n \wedge d\ov z_n
\end{aligned}
\end{equation}                                    
are well-defined nonnegative $(n,n)$-forms on $X$ known respectively as the Carath\'eodory and Kobayashi-Eisenman volume forms on $X$. Moreover, if $k_X$ is positive everywhere (which is the case if $X$ is hyperbolic, see \cite{Aba} for example) then the ratio $c_X(p)/k_X(p)=C_X(p)/K_X(p)$ is a well-defined function of $p \in X$ that is invariant under biholomorphisms and is referred to as the $c/k$-invariant or the quotient invariant. By an application of the Schwarz lemma, $c_X(p)/k_X(p) \leq 1$ for all $p \in X$ and it is well known that if $c_X(p^0)/k_X(p^0)=1$ for some point $p^0 \in X$ then $X$ is biholomorphic to $\mf{B}^n$ (see \cite{Gr-Wu}). Note that the $c/k$-invariant can also be defined with respect to the unit polydisc $\De^n$, simply replacing $\mf{B}^n$ by $\De^n$ in \eqref{v-form-coeff} and \eqref{v-form}, and this variant of the invariant also enjoys the same properties. To distinguish the two variants, we will refer to them as the $\mf{B}^n$-variant and $\De^n$-variant.

We first prove a stability property of the volume forms that will be used in the sequel. In what follows $C_X(p)$ and $K_X(p)$ will always mean the coefficients of $c_X(p)$ and $k_X(p)$ respectively as in \eqref{v-form} in some fixed coordinate chart centred at $p$. The following result concerns the behaviour of the volume forms for monotone sequences of increasing domains:

\begin{lem} \label{L4}
Let $X_j \subset X_{j+1}$ be an exhaustion of $X$. Then $c_{X_j}(\cdot) \to c_X(\cdot)$ and $ k_{X_j}(\cdot) \to k_X (\cdot)$ for both the $\mf{B}^n$-variant and $\De^n$-variant.
\end{lem}

\begin{proof}
We give the proof for the $\De^n$-variant of volume forms. The proof for the $\mf{B}^n$-variant is exactly the same. Fix $p \in X$ and a coordinate chart centred at $p$. The inclusions $X_j \subset X_{j+1} \subset X$ implies that $\{C_{X_j}(p)\}$ is a decreasing sequence with a lower bound $C_{X}(p)$. Therefore, it is enough to show that $C_{X}(p)$ is the infimum of this sequence. Recall that the extremals for the Carath\'{e}odory volume form always exist since the definition involves maps into $ {\Delta}^n$. Let $ f^j: X_j \to \Delta^n $ be the sequence of extremals for $X_j$ at the point $p$ of $X$, i.e.,
\[
\vert \det\left(f^j \right)'(p) \vert^2 = C_{X_j}(p).
\]
The family $ \{f^j\} $ is uniformly bounded and hence there is a subsequence $ \{ f^{j_k} \} $ that 
converges uniformly on compact subsets of $X$ to $f: X \to \overline{\Delta^n}$ with $f(p)=0$. The maximum principle shows that $f(X) \subset \Delta^n $ and hence 
$\vert \det f'(p) \vert^2 \leq C_X(p)$. It follows that for any given $\ep>0$ we have
\[
C_{X_{j_k}} (p) = \vert \det \left( f^{j_k}\right)'(p) \vert^2 \leq \vert f'(p) \vert^2+\ep \leq C_X(p) + \epsilon
\]
for all large $j$ establishing that $C_X(p)=\inf C_{X_j}(p)$ and hence the convergence $\lim_{j \to \infty} C_{X_j}(p) = C_X(p)$ as required.

To work with the Kobayashi-Eisenman volume forms, as before note that $\{K_{X_j}(p)\}$ is a decreasing sequence with a lower bound $K_{X}(p)$, and we show that it is the infimum of this sequence. Let $\ep>0$ and consider a holomorphic mapping  $ f: \Delta^n \to X$ with $ f(0) = p $ that almost realizes $ K_X(p) $, i.e., 
\[
K_X(p) \leq \det \vert f'(0) \vert^{-2} \leq K_X(p) + \epsilon.
\]
Fix $\delta \in (0,1)$ and define the holomorphic mapping $ g : \Delta^n \rightarrow X $ by setting $ g(w) = f (\delta w) $. Since the image $ f (\delta \Delta^n ) $ is compactly contained in $ X $, it follows that $ g : \Delta^n \rightarrow X_j $ for $ j $ large and hence
it follows that
\begin{align*}
K_{X_j}(p) \leq |g'(0)|^{-2} = \delta^{-2n} \vert \det f'(0) \vert^{-2}
\leq \delta^{-2n} \left( K_X(p) + \epsilon \right)
\end{align*}
for $j $ large. Letting $\delta \to 1$ yields
\begin{align} \label{E13}
 \limsup_{j\rightarrow \infty} K_{X_j}(p) \leq K_X(p) + \epsilon,
\end{align}
and hence $K_{X}(p)=\inf K_{X_j}(p)$ as required.
\end{proof}

Now let $ \Omega $ be as in Theorem~\ref{T1} (v) and recall that we are in case (b), i.e., $p^j \to p^0 \in \pa \Om$. It should be noted that each singular boundary point of $\Om$ is a local peak point of $ \Omega $ (see Lemma 5.1 of \cite{Fu-Wong}). Moreover, each smooth boundary point $ p $ is a \textit{local weak peak point} of $ \Omega $, i.e., there exist a neighborhood $ U_p $ of $ p $ and a function $ f_p $ holomorphic on $ \Omega \cap U_p $ and continuous on $ \overline{\Omega} \cap U_p $ such that $ f_p(p) = 1,|f_p (z)| < 1 $ for $ z \in  \Omega \cap U_p $, and $ |f_p (z)| \leq 1 $ for $ z \in \overline{\Omega} \cap U_p $. It follows that $ \Omega $ is complete hyperbolic and hence taut (refer Corollary 3.3 of \cite{Fu} and the following remark therein). We will show that $ M $ is biholomorphic to $ \Delta^2 $. Indeed, $ p^0 \in \partial \Omega $ is either a singular boundary point or a regular boundary point. We consider the $\De^n$-variant of the $c/k$-invariant. Fix a coordinate chart centred at $z^0$. Following the proof of the main theorem of \cite{Fu-Wong}, and making the relevant changes therein, it follows that in both the cases, $ {C_{M_j}(z^0)}/{K_{M}(z^0)} \geq 1 $
for each $ j $ in local coordinates centred at $z^0$. But since $ M_j \subset M_{j+1} $ is an exhaustion of $D $, the above observation together with Lemma \ref{L4} implies that $ {C_{M}(z^0)}/{K_{M}(z^0)} \geq 1 $. Recall that the inequality $ C_M(z^0)/K_M(z^0) (\cdot) \leq 1 $ always holds and therefore, $ {C_{M}(z^0)}/{K_{M}(z^0)} = 1 $. As a consequence, it follows that $ M $ is biholomorphically equivalent to $ \Delta^2$. \qed

We note that the simply connected assumption on $\Om$ is required only when $ p^0 \in \partial \Omega $ is a regular boundary point which can be seen from the proof in \cite{Fu-Wong}. The proof of Theorem \ref{T1} (iv) can be modified to provide an alternate approach for the Union problem in the case $ \Omega $ is strongly pseudoconvex (refer Theorem 1.2 of \cite{Behrens}). More precisely, 
\begin{thm}
Let $ M $ in the Union problem be a hyperbolic manifold of dimension $n$. If $ \Omega $ is a bounded $ C^2$-smooth strongly pseudoconvex domain in $ \mathbb{C}^n $, then $ M $ is biholomorphic either to $ \Omega $ or to $ \mathbb{B}^n $.
\end{thm}

\begin{proof}
For the duration of this proof we will be working with the $\mf{B}^n$-variant of the $c/k$-invariant. Let $ z^0 \in M $ be a fixed point, $ \psi^j : M_j \rightarrow \Omega $ be biholomorphisms from $ M_j $ onto $ \Omega $, and $p^j=\psi^j(z_0)$. Much like before, the following two possibilities need to be considered:
\begin{enumerate}
 \item [(a)] $\{p^j\}$ is compact in $\Omega$, and  
\item [(b)] $\{p^j\} $ has at least one limit point $ p^0 \in \partial \Omega$. 
\end{enumerate}
Notice that $ M $ turns out to be biholomorphically equivalent to $ \Omega $ in case (a) as earlier. In case (b), we show that $ M $ is biholomorphic to $ \mathbb{B}^n $. Indeed, fixing a coordinate chart in $D_1$ centred at $z^0$, it is immediate that
\begin{equation} \label{E14}
\frac{C_{M_j}(z^0)}{K_{M_j}(z^0)}= \frac{C_{\Omega}\left(p^j\right) }{K_{\Omega}\left(p^j\right) }. 
\end{equation}
But since $ M_j \subset M_{j+1} $ is an exhaustion of $M$ and $ M $ is hyperbolic, Lemma~\ref{L4} imply that
\begin{equation} \label{E15}
\frac{C_{M_j}(z^0)}{K_{M_j}(z^0)} \to \frac{C_M(z^0)}{K_M(z^0)}.
\end{equation}
On the other hand, it follows from \cite{Rosay} that 
\begin{equation} \label{E16}
 \frac{C_{\Omega}\left(p^j\right) }{K_{\Omega}\left(p^j\right) } \rightarrow 1.
\end{equation}
Combining (\ref{E13}), (\ref{E14}) and (\ref{E15}), and letting $ j \rightarrow \infty $ gives $ C_M(z^0)= K_M(z^0) $. It follows that $ M $ is biholomorphic to $ \mathbb{B}^n $, and the proof is complete.
\end{proof}

\subsection{Proof of Theorem~\ref{T1} (vi): $ \Omega $ is the symmetrized polydisc $ \mathbb{G}_n $} 
The symmetrized polydisc is the image $ \mathbb{G}_n = \pi (\Delta^n) $ of $\De^n$ under the symmetrization map $ \pi: \mathbb{C}^n \rightarrow \mathbb{C}^n $ defined by 
\begin{equation}\label{strpoly}
\pi (z_1, \ldots, z_n) = \left( \sum_{1 \leq j_1< \cdots < j_k \leq n } z_{j_1} \ldots z_{j_k}\right), \quad 1 \leq k \leq n.
\end{equation}
In particular, $ \mathbb{G}_1= \Delta $ and $ \mathbb{G}_2 $ is called the symmetrized bidisc. It is known that $ \mathbb{G}_n $ is taut, $ \left( \pi \right)^{-1} ( \mathbb{G}_n) = \Delta^n $, and $ \pi \vert_{\Delta^n} : \Delta^n \rightarrow \mathbb{G}_n $ is proper with multiplicity $ n! $.

Recall that we are in case (b), i.e., $p^j \to p^0 \in \pa \Om$. We prove that there exists a proper correspondence from $ M$ to $ \Delta^n $. To this end, consider the symmetrization mapping $ \pi: \Delta^n \rightarrow \Omega $ as defined above and the $ n! $-valued holomorphic mapping $ \pi^{-1} : \Omega \rightarrow \Delta^n $. Set $ \lambda^{j,0} \in \Delta^n $ to be any one of the $ n! $  preimages of $ \psi^j(z^0) $ under $ \pi $. Define the multiple valued mappings
\begin{alignat*}{3}
 H^j \circ \pi^{-1} \circ \psi^j : M_j \rightarrow \Delta^n
\end{alignat*}
where $ H^j $ are automorphisms of $ \Delta^n $ so chosen that $ H^j $ map $ \lambda^{j,0} $ to the origin. Note that the inverses
\[ 
\left( \psi^j \right)^{-1} \circ \pi \circ \left(H^j\right)^{-1} : \Delta^n \rightarrow M_j
\] 
are proper holomorphic mappings. At this point, recall from \cite{PK}, the notion of normality for correspondences and Theorem 3 therein, the version of Montel’s theorem for proper holomorphic correspondences
with varying domains. It is immediate from the construction that the set $ H^j \circ \pi^{-1} \circ \psi^j (z^0)  $ contains the origin. In particular, it is evident that $ \{ H^j \circ \pi^{-1} \circ \psi^j \} $ fails to be compactly divergent.  In this setting, Theorem 3 of \cite{PK} implies that the family $ \{ H^j \circ \pi^{-1} \circ \psi^j \} $ of proper correspondences is normal and hence some subsequence converges to a proper correspondence from $ M $ to $ \Delta^n $ with each fibre having cardinality at most $ n! $. \qed.

\section{When $ M $ is non-hyperbolic}
The purpose of this section is to prove Theorem \ref{T2}. Firstly, recall the biholomorphisms $ \psi^j: M_j \rightarrow \Omega $ and $ \phi^j := \left(\psi^j \right)^{-1} : \Omega \rightarrow M_j $. As in Section $ 2 $, fixing $ z^0 \in D $, and writing $p^j=\psi^j(z^0)$, one needs to consider the following two cases:
\begin{enumerate}
 \item[(a)] $ \{ p^j \} $ is compactly contained in $ \Omega $, and
 \item[(b)] $ \{ p^j \} $ has at least one limit point $ p^0 \in \partial \Omega $.
 \end{enumerate}
Consider case (a). The first step is to construct a retract $Z$ of $\Om$. The tautness of $ \Omega $ forces that $ \{ \psi^j \} $, after passing to a subsequence, converges uniformly on compact subsets of $ M $ to a holomorphic mapping $ \psi: M \rightarrow {\Omega} $. Following \cite{FS1981} and \cite{Behrens}, consider the holomorphic mappings 
\begin{equation*}
 \alpha^j := \psi \circ \phi^j : \Omega \rightarrow \Omega.
\end{equation*}
Observe that
\[
\al^j\big(\psi^j(z^0)\big)=\psi \circ \phi^j \big(\psi(z^0)\big)=\psi(z^0).
\]
Again, exploiting the tautness of $ \Omega $, it is possible to pass to a subsequence of $ \{\alpha^j \} $ that converges uniformly on compact subsets of $\Om$ to a holomorphic mapping $ \alpha: \Omega \rightarrow \Omega$. Moreover,
\begin{equation*}
 \alpha \circ \psi (z) = \lim_{j \rightarrow \infty} \alpha^j \circ \psi^j (z) = \lim_{j \rightarrow \infty} \psi \circ \phi^j \circ \psi^j (z) = \psi(z)
\end{equation*} 
for all $ z \in M$. Following \cite{FS1981} and \cite{Behrens}, define
\begin{equation*}
Z = \{ w \in \Omega : \alpha(w) = w \} \supset \psi(M).
\end{equation*}
Note further that
\[
\al\big(\al^j(w)\big)= \al \circ \psi \circ \phi^j (w) = \psi \circ \phi^j (w) = \al^j(w)
\]
and hence $\al^j$ maps $Z$ to $Z$. Now let $ k $ be the maximal rank of $ {\psi} $ on $ M $. Recall that $ {\alpha} = \lim_{j \rightarrow \infty}  {\psi} \circ {\phi}^j $, and hence rank ${\alpha} \leq  k $. By verbatim arguments from Lemmas 4.1-4.4 of \cite{FS1981} one can check that $Z = \psi (M) $, $Z$ is a closed connected submanifold of $\Om$, the mapping $ \alpha $ is a holomorphic retraction from $ \Omega $ onto $ Z $, and the mapping $\psi$ has constant rank $k$. Also, as $ \psi^j$ are biholomorphisms and hence Kobayashi isometries,
\begin{equation} \label{E8}
 F_{M_j} (p, v) = F_{\Omega} \big( \psi^j(p), d\psi^j(p) v \big), \quad (p,v) \in T^{1,0}M.
\end{equation}
The tautness of $ \Omega $ implies via a normal family argument that $ F_{\Omega}( \cdot, \cdot) $ is jointly continuous and hence letting $ j \rightarrow \infty $ in (\ref{E8}) yields
\begin{equation} \label{E10}
 F_M(p,v) = F_{\Omega} \big( \psi(p), d \psi(p) v \big), \quad (p,v) \in T^{1,0}M.
\end{equation}
Since $\Om$ is hyperbolic,
\begin{equation*}
F_M(p,v) = 0 \; \mbox{ iff } \; d \psi(p) v = 0,
\end{equation*}
and hence the dimension of $\{v \in T_pM: F_M(p,v)=0\}$ is equal to the nullity of $d\psi(p)$ which is the constant $n-k$ as we have noted above. This proves part (i) of the theorem in case (a). It is worthwhile mentioning that the zero set of the infinitesimal Kobayashi metric on $ M $, i.e., the set 
\begin{equation*} 
 \big\{ (p,v) \in T^{1,0}M : F_D(p,v) = 0 \big\},
\end{equation*}
turns out to be a vector bundle over $ M $ of dimension $ (n-k) $. The proofs of parts (ii) and (iii) of Theorem \ref{T2} in case (a) follow exactly as in \cite{FS1981} and are therefore omitted here. 

\medskip

Now consider case (b). We scale $ \Omega $ along $p^j=\psi^j(z^0)$ and recall that $A^j: \Om \to \Om_j$ are the scaling maps, $\tilde{\psi^j} = A^j \circ \psi^j : M_j \rightarrow \Omega_j$ is a normal family, $\ti \psi^j (z^0)=q^j \to q^0 \in \Om$, and $ \tilde{\psi} : M \rightarrow \Omega_{\infty} $ is a limit of $\ti \psi^j$. Moreover, $\ti \phi^j=(\psi^j)^{-1}: \Om_j \to M_j$. Following \cite{FS1981} and \cite{Behrens}, consider the holomorphic mappings 
\begin{equation*}
\tilde{\alpha}^j := \tilde{\psi} \circ \tilde{\phi}^j : \Omega_j \rightarrow \Omega_{\infty}.
\end{equation*}
Observe that
\[
\ti \al^j (q^j)=\tilde{\psi} \circ \tilde{\phi}^j \big(\psi^j(z^0)\big) = \ti \psi (z^0)=q^0.
\]
Again, exploiting the tautness of $ \Omega_{\infty} $, it is possible to pass to a subsequence of $ \{ \tilde{\alpha}^j \} $ that converges uniformly on compact sets of $ \Omega_{\infty} $ to a holomorphic mapping $ \tilde{ \alpha} :
\Omega_{\infty} \rightarrow{\Omega}_{\infty} $. Also,
\begin{equation*}
\tilde{\alpha} \circ \tilde{\psi} (z) = \lim_{j\rightarrow \infty} \tilde{\alpha}^j \circ \tilde{\psi}^j(z) =  \lim_{j\rightarrow \infty} \tilde{\psi} \circ \tilde{\phi}^j \circ \tilde{\psi}^j(z) = \tilde{\psi}(z),
\end{equation*}
for all $ z \in M $. Following \cite{FS1981} and \cite{Behrens}, define
\begin{equation*}
\tilde{Z} = \{ w \in \Omega_{\infty} : \tilde{\alpha}(w) = w \} \supset \tilde{\psi}(D). 
\end{equation*}
Note further that
\[
\ti \al\big(\ti \al^j(w)\big)= \ti \al \circ \ti \psi \circ \ti \phi^j (w) = \ti \psi \circ \ti \phi^j (w) = \ti \al^j(w)
\]
and hence $\ti \al^j$ maps $Z$ to $Z$.  Let $ \ti k $ be the maximal rank of $ \tilde{\psi} $ on $ M $. Recall that $ \tilde{\alpha} = \lim_{j \rightarrow \infty}  
\tilde{\psi} \circ \tilde{\phi}^j $, and hence rank $ \tilde{\alpha} \leq \ti k $. By verbatim arguments from Lemmas 4.1-4.4 of \cite{FS1981} one can check that $\ti Z = \ti \psi (M) $, $\ti Z$ is a closed connected submanifold of $\Om_{\infty}$, the mapping $\ti \alpha $ is a holomorphic retractions from $ \Omega_{\infty} $ onto $ \ti Z $, and the mapping $\ti \psi$ has constant rank $\ti k$. Also, as $\tilde{\psi}^j $ are biholomorphisms and hence Kobayashi isometries,
\begin{equation} \label{E9}
F_{M_j} (p, v) = F_{\Omega_j} \big( \tilde{\psi}^j(p), d\tilde{\psi}^j(p) v \big), \quad (p,v) \in T^{1,0}M.
\end{equation}
Further, appealing to (\ref{E5}) and letting $ j \rightarrow \infty $ in (\ref{E9}) gives
\begin{equation} \label{E11}
 F_M(p,v) = F_{\Omega_{\infty}} \big( \tilde{\psi}(p), d \tilde{\psi}(p) v \big).
\end{equation}
Again, as $\Om_{\infty}$ is hyperbolic, we obtain
\begin{equation*}
F_M(p,v) = 0 \; \mbox{ iff } \; d \tilde{\psi}(p) v = 0.
\end{equation*}
Therefore, the dimension of $\{v\in T_pM :F_M(p,v)=0\}$ is equal to the nullity of $d\ti \psi(p)$ which is the constant $n -\ti k$. We further metion that the zero set of the infinitesimal Kobayashi metric on $ M$, i.e., the set 
\begin{equation*} 
 \{ (p,v) \in T^{1,0} M : F_M(p,v) = 0 \},
\end{equation*}
in this case is a vector bundle over $ M $ of dimension $ (n-\ti k) $. The proofs of parts (ii) and (iii) of Theorem \ref{T2} follow exactly as in \cite{FS1981} and are therefore omitted here.

\section{Retracts of $\Delta \times \mathbb{B}^{n-1}$ -- Proof of Theorem \ref{retracts}}
For notational convenience, we work with $\Delta \times \mathbb{B}^{n}$ instead of $\Delta \times \mathbb{B}^{n-1}$. Before proceeding ahead, recall that a holomorphic retract of a complex manifold $X$ is a complex submanifold of $X$ (see Chapter 2 of \cite{Ce}). We shall refer $X$ and its points as the trivial retracts. Let $\alpha=(\alpha_0,\alpha_1, \ldots, \alpha_n)$ be a retraction of $\Delta \times \mathbb{B}^n$ and $Z$ be the range of $ \alpha $. Note that $d\alpha(p)v= v $ for all $p \in Z$ and $v \in T_pZ$. Moreover, assume that $Z$ contains the origin, i.e., $\alpha(0)=0$, by composing with an automorphism of $ \Delta \times \mathbb{B}^n$, if required.


Further, note that $Z$ is a non-compact subset of $\Delta \times \mathbb{B}^n$ and hence intersects the boundary 
\[
\partial (\Delta \times \mathbb{B}^n)
= (\partial \Delta \times \mathbb{B}^n) \cup (\Delta \times \partial \mathbb{B}^n)
\cup (\partial \Delta \times \partial \mathbb{B}^n).
\]
%
The main idea is to study how $ L = T_0 Z $ intersects $ \partial (\Delta \times \mathbb{B}^n)$.
Writing $ G= \Delta \times \mathbb{B}^n $, $ L' = L \cap G$, and $\partial L'= L \cap \partial G$, the following four cases arise:
\begin{itemize}
\item [(i)] $\pa L' \subset \pa \De \times \pa \mathbb{B}^n$.
\item [(ii)] $\pa L' \subset \partial \Delta \times \overline{\mathbb{B}^n}$.
\item [(iii)] $\pa L' \subset \overline{\Delta} \times \partial \mathbb{B}^n$.
\item [(iv)] $\pa L'$ intersects both $\pa \De \times \mathbb{B}^n$
 and $\De \times \pa \mathbb{B}^n$.
\end{itemize}

Here, and in the sequel, a {\it linear subspace of} $D$ refers to the intercept $l \cap D$ where $l$ is any complex linear subspace of $\mathbb{C}^{n+1}$. Further, the  closed set $l \cap \partial D$ denotes the boundary of such a linear subspace.

\subsection*{Case (i)} Here, both the projections $\pi_1: L' \to \Delta$ and
$\pi_2: L' \to \mathbb{B}^n$ are proper. 
Since proper holomorphic maps do not decrease the dimension, it follows from the properness of $ \pi_1 $ that ${\rm dim}(L) = {\rm dim}(L')=1$, and consequently that, $ Z $ is a one-dimensional complex submanifold of $ \Delta \times \mathbb{B}^n $. Since $L'$ is linear, the properness of $\pi_1$ also means that the fibres of $\pi_1$ are singletons, and hence $\pi_1: L' \rightarrow \Delta$ is a biholomorphism. Now, use the inverse of this map to parametrize $L'$ as 
\[
L' = \Big\{ \big( w, \beta_1(w), \ldots, \beta_n(w)\big) : w \in \Delta \Big\},
\]
where $\beta_1, \ldots, \beta_n$ are linear functions of a single complex
 variable. Recall that $ L' $ contains the origin, and hence $\beta_j(w) = c_jw$ for $c_j \in \mathbb{C}$. By the hypothesis of case (i), it follows that $(c_1, \ldots, c_n) \in \partial \mathbb{B}^n$. 
 
So far, we have a description of $ L'$. The claim is that $L'=Z$. Indeed, pick  $(w_0,z^0) = (w_0,z^0_1,\ldots, z^0_n ) \in \partial L' \subset \partial \Delta \times \partial \mathbb{B}^n$ and 
consider the complex line through this point and the origin. Since $ \mbox{dim}(L') = 1$, this complex line must be $L'$ itself. Thus,
\[
\varphi(t) = (w_0t, z^0_1t, \ldots, z^0_nt)
\]
is a parametrization of $L'$. Restrict the retraction $\alpha$ to $L'$ and consider
\[
g_1(t) = (\pi_1 \circ \alpha \circ \varphi) (t) = \alpha_0(w_0t, z^0_1t, \ldots, z^0_nt)
\]
and 
\[
g_2(t) = (\pi_2 \circ \alpha \circ \varphi)(t) = \left( \alpha_1(w_0t,z^0t), \ldots, \alpha_n(w_0t,z^0t) \right).
\]
Note that $g_1: \Delta \rightarrow \Delta $ and $g_2 : \Delta \rightarrow \mathbb{B}^n $. Further, observe that  $g_1(0)=0$ and  
\[
dg_1(0)= d\pi_1 \vert_{0}\circ d \alpha\vert_{0} \circ d \varphi(0) = w_0 \in \partial \Delta
\]
since $ d \varphi(0) = (w_0,z^0) \in T_0Z$, $d\alpha(0)v= v $ for all $v \in T_0Z$. Applying the Schwarz Lemma to $ g_1$ yields that $\alpha_0(w_0t,z^0t)= w_0t$, and hence $ {\alpha_0}(w,z) = w $ on $ L' $. Similarily,  $g_2(0)=0$ and $dg_2(0) = z^0 \in \partial \mathbb{B}^n$. By the Schwarz Lemma (Theorem 8.1.3 of \cite{R}), it follows that $g_2(t) = z^0 t$ or equivalently that
\[
\alpha_j \left(w_0t, z^0_1t, \ldots, z^0_nt \right) = z^0_j t,
\]
for each $ j = 1, \ldots, n $. Hence, $\alpha$ is the identity mapping on $L'$. In particular, it follows that $L' \subset Z$. Apply the identity principle (see Proposition 1, Section 5.6 of \cite{C} for details) to 
conclude that $ L' = Z $. Thus,
\[
Z=\left\{ (w,c_1w,c_2w, \ldots, c_n w) : w \in \Delta\right \}
\]
where $(c_1,\ldots,c_n) \in \partial \mathbb{B}^n $. 
 
\subsection*{Case (ii)} In this case, the projection $\pi_1: L' \to \Delta$ is proper. Using similar arguments as before, it follows that ${\rm dim} (Z) = {\rm dim} (L)={\rm dim} (L')=1$ and 
\[
L' = \left \{ (w, c_1w,c_2w, \ldots, c_n w) : w \in \Delta \right\},
\]
for $(c_1, \ldots, c_n) \in \partial \mathbb{B}^n$.
Pick a point $(w_0,z^0) \in \partial L' \subset \partial \Delta \times \overline{\mathbb{B}^n}$, and consider the restriction of $\alpha$ to the complex line segment joining $(w_0,z^0) $ and the origin. Applying the Schwarz lemma, as in Case (i), shows that $\alpha_0(w,z) = w$ on $ L' $. This implies that $\alpha(L')$ is the graph of the $ \mathbb{B}^n$-valued holomorphic mapping 
\[
 \Delta \ni w \longmapsto \big( \alpha_1(w, c_1w,\ldots, c_n w), \ldots, \alpha_n(w, c_1w,c_2w, \ldots, c_n w)\big).
\]
In particular, $\alpha(L')$ is a one-dimensional subvariety of $Z$. In this setting, the identity principle for the complex analytic sets ensures that $\alpha(L')=Z$, i.e., 
\[
Z=\Big\{ \big( w, \alpha_1(w, c_1w,\ldots, c_n w), \ldots, \alpha_n(w, c_1w,c_2w, \ldots, c_n w) \big) : w \in \Delta \Big\}.
\]

\subsection*{Case (iii)} In this case, the projection $\pi_2 : L' \to \mathbb{B}^n $ is proper. Since proper holomorphic maps do not decrease the dimension, it follows that ${\rm dim}( L') $ is at most $n$. 

\medskip

Suppose that ${\dim(L')} =1$. Let $(w_0, z^0) $ be a point of $ \partial L'$ chosen so that $ (w_0, z^0) \in \Delta \times \partial \mathbb{B}^n $ (if $ w_0 \in \partial \Delta $, then we are in Case (i)). As before, since ${\dim(L')} =1$, the complex line through this point $(w_0, z^0) $ and the origin is $ L' $ and \[
\varphi(t) = (w_0t, z^0_1t, \ldots, z^0_nt)
\]
gives a parametrization of $L'$. Moreover, $g_2: \Delta \rightarrow \mathbb{B}^n $ defined by setting $ g_2(t) =   (\pi_2 \circ \alpha \circ \varphi)(t)$ satisfies $g_2(0)=0$ and $dg_2(0)=z^0$. Apply Theorem 8.1.3 of \cite{R}, as before, to get that 
 $g_2(t)=z^0t $, i.e.,  
 \[
 \Big( \alpha_1 \big( \varphi(t) \big), \ldots, \alpha_n \big(\varphi(t) \big) \Big) = (z^0_1 t, \ldots,  z^0_nt),
\]
which exactly means that
\begin{equation} \label{alphonL}
\big( \alpha_1(w,z_1,\ldots,z_n), \ldots, \alpha_1(w,z_1,\ldots,z_n) \big) = (z_1, \ldots, z_n) \mbox{ on } L'.
\end{equation}
Furthermore, since $L'$ is linear and $\pi_1$ is proper, each fibre of $ \pi_1 $ is a singleton. Hence, $\pi_2$ maps $L'$ biholomorphically onto $\pi_2(L') $, $\pi_2(L') $ is a one-dimensional linear subspace of $\mathbb{B}^n $
and
\[
 L'=\Big\{ \big( \beta_0( z_1, \ldots, z_n), z_1, \ldots, z_n \big) : (z_1, \ldots, z_n ) \in \pi_2(L')\Big\},
\]
where $ \beta_0 $ is holomorphic in $ z_1, \ldots, z_n $. 
It follows that $ \alpha(L') $ is the graph of a $ \Delta $-valued holomorphic function over $ \pi_2(L') $, i.e.,
\[
 \alpha(L') = \Big\{ \Big( \alpha_0\big(\beta_0(z_1,\ldots, z_n) \big),z_1, \ldots, z_n \Big): (z_1, \ldots, z_n ) \in \pi_2(L') \Big\}.
 \]
In particular, $\alpha(L')$ is a one-dimensional analytic variety of $ Z$ and hence $\alpha(L')=Z$ by the identity principle. 

\medskip

Next, suppose that ${\rm dim}(L') = k >1$. Apply the above analysis to each one-dimensional complex linear subspace $l'$ of $L'$ with $ \partial l' \subset \overline{\Delta} \times \partial \mathbb{B}^n$ to 
conclude that 
\begin{equation*} 
\big( \alpha_1(w,z_1,\ldots,z_n), \ldots, \alpha_1(w,z_1,\ldots,z_n) \big) = (z_1, \ldots, z_n) \mbox{ on } l'.
\end{equation*}
It follows that 
\begin{equation*} 
\big( \alpha_1(w,z_1,\ldots,z_n), \ldots, \alpha_1(w,z_1,\ldots,z_n) \big) = (z_1, \ldots, z_n) \mbox{ on } L'.
\end{equation*}
On the other hand, since $\pi_2: L' \to \mathbb{B}$ is proper and $L'$ is linear, it follows that
$\pi_2 $ maps $L'$ biholomorphically onto $ \pi_2(L')$, 
$\pi_2(L') $ is a $k$-dimensional linear subspace of $\mathbb{B}^n $
and hence
\[
 L'=\Big\{ \big( \beta( z_1, \ldots, z_n), z_1, \ldots, z_n \big) : (z_1, \ldots, z_n ) \in \pi_2(L')\Big\},
\]
for some holomorphic function $ \beta $.
As a consequence, 
\[
 \alpha(L') = \left\{ \Big( \alpha_0\big(\beta(z_1,\ldots, z_n) \big),z_1, \ldots, z_n \Big): (z_1, \ldots, z_n ) \in \pi_2(L') \right\},
 \]
or equivalently that, $ \alpha(L') $ is the graph of a $ \Delta $-valued holomorphic function over $ \pi_2(L') $. 
It is immediate that $\alpha(L')$ is a $k$-dimensional analytic variety of $ Z$. Since $ {\rm dim}(Z) = {\rm dim}(L') = k $, it follows that $\alpha(L')=Z$ as before.

\subsection*{Case (iv)} Here, $\partial L'$ intersects 
both $\partial \Delta \times \mathbb{B}^n$ and 
$\Delta \times \partial \mathbb{B}^n$.

\medskip

Let $ l $ denote the complex line joining the origin and a point $ (w_0, z^0) \in \partial L' \cap \left(\partial \Delta \times \mathbb{B}^n \right) $.
Then the line $ l $ does not intersect $ \Delta \times \partial \mathbb{B}^n $. It follows that $ l $ is properly contained in $ L'$ and hence ${\rm dim}(L') \geq 2$. 

\medskip

Note that open pieces of $ \partial L' $ are contained in $\partial \Delta \times \mathbb{B}^n$ and 
$\Delta \times \partial \mathbb{B}^n$. Indeed, if $\tilde{l}$ is a complex line in $L'$ that intersects $\partial \Delta \times \mathbb{B}^n$ (or $\Delta \times \partial \mathbb{B}^n$), then the lines obtained by small perturbations of $ \tilde{l} $ intersect the open piece $\partial \Delta \times \mathbb{B}^n$ (or  
$\Delta \times \partial \mathbb{B}^n$ respectively) of $ \partial G$. Further, the boundary of the intercept of each such line with $G$ is entirely contained in $\partial \Delta \times \mathbb{B}^n$ (or $\Delta \times \partial \mathbb{B}^n$). Then the union of these lines contains an open subset $U_L$ of $L'$.

\medskip
 
Next, applying the Schwarz lemma arguments as in Case (ii) and Case (iii) 
yields that 
 $ {\alpha_0}(w,z)=w $
and
 $\left( \alpha_1(w,z), \ldots, \alpha_1(w,z) \right) = z $
on $ L'$. Hence, $\alpha$ is the identity mapping on $L'$. As a consequence, $ L' \subset Z $. But ${\rm dim}(L')={\rm dim}(Z)$, so that $L'=Z$ as before. This completes the proof of Theorem~\ref{retracts}. 



\medskip

Combining the above observations with Theorem 3.2 and the Main Theorem of \cite{FS1981}, we get the following result:

\begin{thm}
Let $ \Omega= \Delta \times \mathbb{B}^{n-1}$ in the union problem. If $M$ is hyperbolic, then 
$M $ is biholomorphic to $ \Delta \times \mathbb{B}^{n-1}$. If $ M $ is non-hyperbolic and the corank of $F_M$ is one, then $M$ is biholomorphic to $Z \times \mathbb{C}$
 where $Z$ is a retract of $\Delta \times \mathbb{B}^{n-1}$ as given by  Theorem~\ref{retracts}. Moreover,  $Z$ is either the graph of a $\Delta$-valued holomorphic function on $\mathbb{B}^{n-1}$ or a linear subspace of $\Omega$ of (complex) dimension at least two.
\end{thm}

\section{Can $M$ be quasiprojective?--Proof of Theorem~\ref{qp}} 
Let $M=\mf{P}^n \setminus Z$ where $Z$ is a divisor in $\mf{P}^n$. Firstly, we show that $M$ cannot be hyperbolic. Suppose not; then Theorem 3.2 of \cite{FS1981} provides a biholomorphism $\phi: M\to\Om$. Then, each of the components of $\varphi$ are bounded and by the Riemann removable singularities theorem, applied to these component functions says that $Z$ is removable for all these component functions and yields their holomorphic extendibility to $\mathbb{P}^n$ and thereby for $\varphi$ as well. However, as
$\mathbb{P}^n$ is a compact complex manifold, it follows by Liouville's theorem that $\varphi$ is constant which is impossible, as $\varphi$ is a biholomorphic map on $M$. Therefore, $M$ cannot be a hyperbolic.

\medskip

Next, let $\Om=\mf{B}^n$ or $\De^n$ and as just noted, $M$ is not hyperbolic. If possible, assume that the corank of $F_M$ is one. In this case, we first recall Corollary 4.8 of \cite{FS1981} according to which the domain is biholomorphic to $\Om'\times \mf{C}$ where $\Om'=\mathbb{B}^{n-1}$ or $\De^{n-1}$ respectively. Let 
$\pi: \Om'  \times \mf{C} \to\Om'$ be the natural projection. Now, consider the map
$\tilde{\varphi} = \pi \circ \varphi$ which is a mapping firstly from $M$ to $\Om'$; but then as before,  
note that the Riemann removable singularities theorem, when applied to the components of $\tilde{\varphi}$ yields their holomorphic extendibility to $\mathbb{P}^n$, and thereby of $\tilde{\varphi}$ as well. Now $\mathbb{P}^n$
 being a compact complex manifold forces 
$\tilde{\varphi}$ to be constant. What this means for our map $\varphi$ is that it maps all of $M$ into just one of the fibres of $\pi$ which are all one-dimensional; but this is impossible as $\varphi$ is a biholomorphic mapping on 
$M$ (whose dimension is $n \geq 2$; the case $n=1$ is no exception, the arguments are only even more easier). We conclude therefore that if at all  a quasiprojective variety $M$ can be exhausted by biholomorphic images of $\mf{B}^n$ or $\De^n$, then the corank of its Kobayashi metric must be strictly bigger than one.

\section{Can $M$ be a co-pluripolar set?--Proof of Theorem~\ref{cp}}
If possible, assume that $M$ is hyperbolic. Then, Theorem 3.2 of \cite{FS1981} applies to give a biholomorphism $\psi:M \to \mathbb{B}^{n}$. As $\psi$ is bounded and $P$ is pluripolar, $\psi$ extends to a holomorphic map on all of $M$. But then $\psi(P) \subset \pa \mf{B}$ and the maximum principle forces $\psi$ to be a constant which is a contradiction. We do not know whether it is possible for the corank to be one in this case. This would amount to showing that $M = X\setminus P$ cannot be biholomorphic to $\Delta \times \mbb C$.

\section{Some concluding remarks}
It is possible to formulate versions of Theorems~\ref{T1} and \ref{T2} with only local assumptions on the model domain $\Om$; the proofs, being similar, are omitted here.

\begin{thm} \label{T3}
Let a hyperbolic manifold $M$ be the union of an increasing sequence of open subsets $M_j$, for each of which there exists a biholomorphism $\psi^j:M_j \to \Omega$ where $\Omega$ is a bounded taut domain in $\mathbb{C}^n$. If there exists a point $z^0 \in M$ with $\{\psi^j(z^0)\}$ being compactly contained in $\Omega$, then $M$ is biholomorphic to $\Omega$. If not, for any $z^0 \in M$, the sequence $\{\psi^j(z^0)\}$ has at least one limit point $p^0 \in \partial \Omega$. 

\begin{enumerate}
\item[(i)] If $p^0$ is a Levi corank one point of $\partial \Omega$, then $ M $ is biholomorphic to a limiting domain of the form
\begin{equation*}
 \Omega_{\infty}=\Big\{ z \in \mathbb{C}^n : 2 \Re z_n + P_{2m} \left(z_1, \overline{z}_1 \right) + \sum_{j=2}^{n-1} \vert z_j \vert^2 < 0 \Big\},
\end{equation*}
where $ 2m $ is the $1$-type of $ \partial \Omega $ at $ p^0 $ and $ P_{2m} \left(z_1, \overline{z}_1 \right) $ is a subharmonic polynomial of degree at most $ 2m $ without any harmonic terms.
 
\item[(ii)] If $\partial \Omega$ near $p^0$ is smooth and convex, then $ M $ is biholomorphic to a complete hyperbolic convex domain in $\mathbb{C}^n$. Further, if $p^0$ is a point of finite type $ 2m $, then $M$ is biholomorphic to a limiting domain
\begin{equation*}
 \Omega_{\infty}=\Big\{ z \in \mathbb{C}^n : 2 \Re z_n + P \left('z, \overline{'z} \right) < 0 \Big\},
 \end{equation*}
where $P \left('z, \overline{'z} \right)$ is a convex polynomial of degree at most $ 2m $ without any harmonic terms. If $\partial \Omega$ near $ p^0 $ is 
convex (but not necessarily smooth) and does not contain any nontrivial complex analytic varieties, and $M$ is (a priori) known to 
be complete hyperbolic, then $M$ is biholomorphic to $\Omega$.

\item[(iii)] If $n=2$ and $ \Omega $ near $p^0 $, after a holomorphic change of coordinates, is a strongly pseudoconvex polyhedral domain, then $M$ is biholomorphic to a limiting domain $\Om_{\infty}$, which is one of $ \Delta^2 $, $ \mathbb{B}^2 $, or a Siegel domain as in \eqref{Sie}.

\end{enumerate}
\end{thm}

As a consequence, note that the $1$-type of every boundary accumulation point of $\{\psi^j(z^0)\}$ for any $ z^0 \in M $
must be the same. In particular, the Levi geometry of  $ \partial \Omega $ cannot be different for different boundary accumulation points. 

The second part of (ii) above, i.e., when $M$ is complete hyperbolic, implies that if $p^0$ has a neighbourhood $U$ such that $\Omega \cap U$ is biholomorphic to the minimal ball for instance, then the only possibility is that $M $ is biholomorphic to $ \Omega$.

\begin{thm}
Let a Kobayashi corank one manifold $M$ be the union of an increasing sequence of open subsets $M_j$, for each of which there exists a biholomorphism $\psi^j:M_j \to \Omega$ where $\Omega$ is a bounded taut domain in $\mathbb{C}^n$. If there exists a point $z^0 \in M$ with $\{\psi^j(z^0)\}$ being compactly contained in $\Omega$, then $M$ is biholomorphic to a fibre bundle with fibre $\mf{C}$ over a retract $Z$ of $\Om$. If not, for any $z^0 \in M$, the sequence $\{\psi^j(z^0)\}$ has at least one limit point $p^0 \in \partial \Omega$. Then under the hypothesis (i), (ii), or (iii) of Theorem~\ref{T3}, $M$ is biholomorphic to a fibre bundle with fibre $\mf{C}$ over a retract $Z$ of the corresponding limiting domain $\Om_{\infty}$.
\end{thm}

 Finally, we do not know whether a version of Theorem \ref{T2} holds when $\Omega$ is either the symmetrized polydisc or a simply connected domain with generic piecewise smooth Levi-flat boundary.


\end{document}